\documentclass[preprint,10pt]{elsarticle}

\usepackage{amsmath,amsfonts,amssymb,amsthm}
\usepackage{xcolor}
\usepackage[algo2e, ruled, linesnumbered]{algorithm2e}
\usepackage{doi}
\usepackage{graphicx}
\usepackage{afterpage}
\usepackage[protrusion=true,tracking=false,kerning=true]{microtype}

\title{Rational Dual Certificates for Weighted Sums-of-Squares Polynomials with Boundable Bit Size \tnoteref{t1}}
\tnotetext[t1]{This article is part of the volume titled ``Computational Algebra and Geometry: A special issue in memory and honor of Agnes Szanto''.}

\hypersetup{
  pdftitle={Rational Dual Certificates for Weighted Sums-of-Squares Polynomials with Boundable Bit Size},
  pdfauthor={Maria M. Davis, D{\'a}vid Papp}
}

\author[1]{Maria M. Davis}
\author[1]{D{\'a}vid Papp\corref{cor1}\fnref{fn1,fn2}}
\ead{dpapp@ncsu.edu}

\affiliation[1]{organization={North Carolina State University, Department of Mathematics},
addressline={\mbox{Campus Box 8205}},
city={Raleigh},
postcode={NC 27695},
country={USA}}

\cortext[cor1]{Corresponding author. E-mail: \texttt{dpapp@ncsu.edu}}
\fntext[fn1]{Dedicated to Ágnes: a most generous and dedicated mentor, and a wonderful colleague. You will be missed.}
\fntext[fn2]{This material is based upon work supported by the National Science Foundation under Grant No. DMS-1847865.}

\allowdisplaybreaks

\newcommand{\deletethis}[1]{}
\newcommand{\revision}[1]{{#1}}

\DeclareMathOperator{\cond}{cond}
\DeclareMathOperator{\bd}{bd}
\DeclareMathOperator{\tr}{tr}

\newcommand{\be}{\begin{enumerate}}
\newcommand{\ee}{\end{enumerate}}
\newcommand{\bi}{\begin{itemize}}
\newcommand{\ei}{\end{itemize}}
\newcommand{\zz}{\mathbb{Z}}
\newcommand{\rr}{\mathbb{R}}
\newcommand{\nn}{\mathbb{N}}

\newcommand{\R}{\mathbb{R}}
\newcommand{\bS}{\mathbb{S}}
\newcommand{\Oh}{\mathcal{O}}

\newcommand{\Sc}{\Sigma^\circ}      
\newcommand{\Ssc}{(\Sigma^*)^\circ} 
\newcommand{\defeq}{\ensuremath{\overset{\mathrm{def}}{=}}}  

\newcommand{\cC}{\mathcal{C}}
\newcommand{\cP}{\mathcal{P}}

\newcommand{\cV}{\mathcal{V}}
\newcommand{\vA}{\mathbf{A}}

\newcommand{\vd}{\mathbf{d}}
\newcommand{\ve}{\mathbf{e}}

\newcommand{\vI}{\mathbf{I}}
\newcommand{\vM}{\mathbf{M}}

\newcommand{\vp}{\mathbf{p}}
\newcommand{\vq}{\mathbf{q}}
\newcommand{\vs}{\mathbf{s}}
\newcommand{\vS}{\mathbf{S}}
\newcommand{\vt}{\mathbf{t}}
\newcommand{\vu}{\mathbf{u}}
\newcommand{\vv}{\mathbf{v}}
\newcommand{\vw}{\mathbf{w}}
\newcommand{\vx}{\mathbf{x}}
\newcommand{\vy}{\mathbf{y}}
\newcommand{\vb}{\mathbf{b}}
\newcommand{\vz}{\mathbf{z}}
\newcommand{\vzero}{\mathbf{0}}
\newcommand{\vone}{\mathbf{1}}
\newcommand{\T}{\mathrm{T}}

\newtheorem{example}{Example}[section]
\newtheorem{theorem}[example]{Theorem}
\newtheorem{corollary}[example]{Corollary}
\newtheorem{proposition}[example]{Proposition}
\newtheorem{lemma}[example]{Lemma}
\newtheorem{definition}[example]{Definition}
\newtheorem{remark}[example]{Remark}

\begin{document}
\begin{abstract} In (Davis and Papp, 2022), the authors introduced the concept of dual certificates of (weighted) sum-of-squares polynomials, which are vectors from the dual cone of weighted sums of squares (WSOS) polynomials that can be interpreted as nonnegativity certificates. This initial theoretical work showed that for every polynomial in the interior of a WSOS cone, there exists a rational dual certificate proving that the polynomial is WSOS. In this article, we analyze the complexity of rational dual certificates of WSOS polynomials by bounding the bit sizes of integer dual certificates as a function of parameters such as the degree and the number of variables of the polynomials, or their distance from the boundary of the cone. After providing a general bound, we explore several special cases, such as univariate polynomials nonnegative over the real line or a bounded interval, represented in different commonly used bases. We also provide an algorithm which runs in rational arithmetic and computes a rational certificate with boundable bit size for a WSOS lower bound of the input polynomial.
\end{abstract}
\maketitle

\allowdisplaybreaks

\section{Introduction} 
It is well known that nonnegative polynomials with rational coefficients in the interior of sum-of-squares cones are \emph{sums of rational squares}; that is, they have sum-of-squares decompositions that are expressed entirely in terms of rational coefficients and can be verified using rational arithmetic \cite{Powers2011}. The complexity of these rational certificates of nonnegativity can be measured by the bit size of the largest magnitude coefficient in the decomposition; bounding the complexity of the ``simplest’’ certificate and establishing its dependence on relevant parameters such as the degree, the number of variables, or the polynomial’s distance from the boundary of the cone are major open questions. The corresponding algorithmic question is how efficiently these decompositions can be computed in rational arithmetic. Surprisingly, polynomial-time algorithms are difficult to design, and tight complexity bounds of known sum-of-squares decomposition algorithms are hard to come by even in the univariate case \cite{MagronSafeyElDinSchweighofer2019}.

The recent paper by Magron and Safey El Din \cite{MagronSafeyElDin2021-corrected} gives an in-depth review of the state-of-the-art on the complexity of deciding and certifying the nonnegativity or positivity of polynomials over basic semialgebraic sets using SOS certificates; we only recall a few highlights.

The paper \cite{MagronSafeyElDinSchweighofer2019} focuses on univariate nonnegative polynomials over the real line. The most efficient algorithm they analyze returns SOS certificates of bit size $\mathcal{O}(d^4 + d^3\tau)$, wherein $d$ is the degree of the polynomial and $\tau$ is the bit size of the largest magnitude coefficient; a slight improvement over the results in \cite{BoudaoudCarusoRoy2008}, which consider (P\'olya-type) WSOS certificates of positivity for univariate polynomials over $[-1,1]$. These algorithms cannot be generalized to the multivariate case. The multivariate case is considered first in \cite{MagronSafeyElDin2021}, and substantially corrected in the report \cite{MagronSafeyElDin2021-corrected}, by analyzing the bit complexity of a hybrid numerical-symbolic algorithm that recovers exact rational WSOS decompositions from an approximate (numerical) WSOS decomposition of a suitably perturbed polynomial. \revision{The main result in the corrected manuscript is that for $n$-variate SOS polynomials of degree $d$, the coefficients in the SOS certificates have bit sizes of order $\Oh\left(\tau d^{d^{\Oh(n)}}\right)$.}


In this paper, we study these questions in the context of \emph{dual certificates}. Dual certificates were introduced in \cite{DavisPapp2022} by the authors, motivated by (and building on) the duality theory of convex conic optimization, which has seen a number of recent applications in real algebraic geometry \cite{KatthanNaumannTheobald2021,Papp2023}. They are rational vectors from the dual cone of WSOS polynomials that by definition can be represented as a vector with far fewer components than a conventional WSOS decomposition: their dimension is independent of the number of weights, and they avoid the explicit representation of the large positive semidefinite Gram matrices that characterize conventional SOS decompositions. In \cite{DavisPapp2022}, it was established that polynomials in the interior of a WSOS cone have rational dual certificates, also providing new elementary proofs of Powers’s theorems from \cite{Powers2011}. Following up on this work, we now study the bit sizes of the components of dual certificates, as well as exact-arithmetic algorithms for the computation of rational dual certificates.

In the first part of the paper, we show that dual certificates can be rounded (trivially, componentwise) to ``nearby’’ rational dual certificates with computable, ``small'' denominators. This follows from a quantitative version of a property of dual certificates that (in contrast to conventional WSOS decompositions) every WSOS polynomial has a full-dimensional cone of dual certificates. In turn, these rational certificates can be converted to integer dual certificates with boundable bit size. In Section~\ref{sec:gradbounds}, we establish our general results, which are applicable to any WSOS cone certifying nonnegativity over arbitrary basic, closed semialgebraic sets, including unbounded ones. We then provide refinements for the most frequently studied and applied special cases, including univariate polynomials over the real line and over bounded intervals in Section \ref{sec:gradcerts-bases}. \revision{E.g., for univariate polynomials over the real line, we show that every positive polynomial with integer coefficients of bit size at most $\tau$ (in the monomial basis) has an integer dual certificate whose components are of bit size $\mathcal{O}\left(d\tau + d\log(d)\right)$---an improvement from the aforementioned result of \cite{MagronSafeyElDinSchweighofer2019} and from the $n=1$ special case of the bounds obtained in \cite{MagronSafeyElDin2021-corrected}.}

In the second part of the paper (Section \ref{sec:computing}), we provide an algorithm that takes a polynomial with rational coefficients as its input, and computes a sequence of rational lower bounds converging to the optimal sum-of-squares lower bound, along with a corresponding sequence of rational dual certificates certifying these bounds.
The algorithm is based on the one proposed in \cite{DavisPapp2022}, which is an almost entirely numerical hybrid method. Although the method in \cite{DavisPapp2022} is capable of computing rational lower bounds and dual certificates via floating point computations, it is limited by the precision of the floating point arithmetic, and the bit sizes of the computed certificates cannot be bounded.
The new algorithm proposed in this paper, Algorithm \ref{alg:Newton2}, runs entirely in infinite precision (rational) arithmetic. We show that all intermediate computations can be carefully rounded to nearby rational vectors with small denominators in each step, while still maintaining the property that the algorithm converges $q$-linearly to the optimal weighted sums-of-squares lower bound.

\subsection{Preliminaries}\label{sec:preliminaries}
Here, we cover notation and background that we will use throughout the rest of this paper. 

\subsubsection{Weighted SOS polynomials and positive semidefinite matrices}

Recall that a convex set $K\subseteq\R^n$ is called a \emph{convex cone} if for every $\vx\in K$ and $\lambda \geq 0$ scalar, the vector $\lambda\vx$ also belongs to $K$. A convex cone is \emph{proper} if it is closed, \emph{full-dimensional} (meaning $\operatorname{span}(K)=\R^n$), and \emph{pointed} (that is, it does not contain a line). We shall denote the interior of a proper cone $K$ by $K^\circ$ and the boundary of a proper cone $K$ by $\bd(K)$. 

The dual of a convex cone $K\subseteq\R^n$ is the convex cone $K^*$ defined as
\[ K^* = \left\{ \vy \in \R^n\,|\,\forall\vx\in K: \vx^\T\vy \geq 0\right\}. \]

\paragraph{Sum-of-squares (SOS) polynomials} Let $\mathcal{V}_{n,2d}$ denote the cone of $n$-variate polynomials of degree $2d$. We say that a polynomial $p \in \mathcal{V}_{n,2d}$ is \emph{sum-of-squares} (SOS) if there exist polynomials $q_1,\dots,q_k \in \mathcal{V}_{n,d}$ such that $p = \sum_{i=1}^kq_i^2$. Define $\Sigma_{n,2d}$ to be the cone of $n$-variate SOS polynomials of degree $2d$. The cone $\Sigma_{n,2d}\subset \mathcal{V}_{n,2d} \equiv \R^{\binom{n+2d}{n}}$ is a proper cone for every $n$ and $d$. Throughout, we will identify polynomials with their coefficients vectors (typeset bold) in a basis that is clear from the context (but not necessarily in the monomial basis), e.g., $\vt$ for the polynomial $t(\cdot)$ and $\vone$ for the constant one polynomial.

\paragraph{Weighted sum-of-squares} More generally, let $\vw = (w_1,\dots,w_m)$ be some given nonzero polynomials and let $\mathbf{d} = (d_1,\dots,d_m)$ be a nonnegative integer vector. We denote by $\cV_{n,2\mathbf{d}}^\vw$ the space of polynomials $p$ for which there exist $r_1 \in \mathcal{V}_{n,2d_1}, \dots, r_m \in \cV_{n,2d_m}$ such that $p = \sum_{i=1}^m w_ir_i$. A polynomial $p \in\mathcal{V}_{n,2\mathbf{d}}^\vw$ is said to be \emph{weighted sum-of-squares} (WSOS) if there exist $\sigma_1 \in \Sigma_{n,2d_1}, \dots, \sigma_m \in \Sigma_{n,2d_m}$ such that $p = \sum_{i=1}^mw_i\sigma_i$. It is customary to assume that $w_1=1$, that is, the ordinary ``unweighted'' sum-of-squares polynomials are also included in the WSOS cones. Let $\Sigma_{n,2\mathbf{d}}^\vw$ denote the set of WSOS polynomials in $\cV_{n,2\mathbf{d}}^\vw$. \revision{This is nearly identical to the notion of the \emph{truncated quadratic module}, except that the degree of each SOS polynomial is independently selected, rather than by ``truncating'' to a desired total degree. In this manner, $\Sigma_{n,2\mathbf{d}}^\vw$ is automatically a full-dimensional convex cone in the ambient space $\mathcal{V}_{n,2\mathbf{d}}^\vw$ by definition.} Additionally, under mild conditions, the cone $\Sigma_{n,2\mathbf{d}}^\vw$ is closed and pointed; for example, it is sufficient that the set
\begin{equation}\label{eq:Swdef}
S_\vw \defeq \{\vx\in\R^n\,|\,w_i(\vx)\geq 0,\,i=1,\dots,m\}
\end{equation}
is a unisolvent point set for the space $\mathcal{V}_{n,2\mathbf{d}}^\vw$ \cite[Prop.~6.1]{PappYildiz2019}. (A set of points $S\subseteq\rr^n$ is \emph{unisolvent} for a space of polynomials $\mathcal{V}$ if every polynomial in $\mathcal{V}$ is uniquely determined by its function values at $S$.) In particular, this implies that both $\Sigma_{n,2\vd}^\vw$ and its dual cone have non-empty interiors, a crucial assumption throughout the paper.

\paragraph{WSOS polynomials and positive semidefinite matrices}
We will denote the set of $n\times n$ real symmetric matrices by $\bS^n$, and the cone of positive semidefinite $n\times n$ real symmetric matrices by $\bS^n_+$. When the dimension is clear from the context, we use the common shorthands $\vA\succcurlyeq 0$ to denote that the matrix $\vA$ is positive semidefinite and $\vA\succ 0$ to denote that the matrix $\vA$ is positive definite. 

\revision{
It is well-known and easily seen that a polynomial $s$ belongs to $\Sigma_{n,2d}$ if and only if
\[ s(\cdot) = v_{n,d}(\cdot)^\T \vS v_{n,d}(\cdot), \]
wherein $v_{n,d}$ denotes the vector of $n$-variate monomials up to degree $d$, and $\vS \in \bS_+^L$ with $L = \binom{n+d}{d}$ is the \emph{Gram matrix} of the SOS polynomial $s$. This functional equality can be expressed coefficient-by-coefficient, identifying the polynomials on both sides of the equation with their coefficient vectors in a fixed basis. For example, if $n=1$ and both polynomials are represented in the monomial basis, we obtain the classic result that $s(t) = \sum_{i=0}^{2d} s_i t^i$ is SOS if and only if there exists a matrix $(S_{jk})_{j,k=0,\dots,d}$ such that $s_i = \sum_{(j,k):i=j+k}S_{jk}$ for each $i$. More generally, every SOS cone $\Sigma_{n,2d}$ is a linear image of the cone $\bS^L_+$, and if we fix a basis for $\cV_{n,2d}$ and a basis for $\cV_{n,d}$, there is an explicitly computable, surjective, \emph{linear} map $\Lambda^*$ from Gram matrices (positive semidefinite matrices) to coefficient vectors of SOS polynomials. From the dual perspective, it is also well-known (and is equivalent to the above statements) that the dual cone $\Sigma_{n,2d}^*$ is a linear pre-image of $\bS^L_+$. More precisely, there exists an injective linear map $\Lambda:\Sigma_{n,2d}^* \to \bS_+^L$. In the context of algebraic geometry and moment theory, $\Lambda(\vy)$ is called the \emph{truncated moment matrix} of the (pseudo-)moment vector $\vy$, and the map $\Lambda^*$ in the representation of the SOS cone is simply the adjoint of $\Lambda$. E.g., in the univariate example above, $\Lambda(\vy)$ is the Hankel matrix of the vector $\vy$.

Although everything in the previous paragraph generalizes from SOS cones to the WSOS case, the conventional notation and terminology involving moment and localizing matrices is rather cumbersome, and is largely unnecessary for this paper. To follow the rest of the paper, it is sufficient to keep in mind that regardless of the number of variables $n$, the degree vector $\vd$, the choice of weights $\vw$, and the polynomial bases used to represent the polynomials of various degrees, \emph{the WSOS cone $\Sigma_{n,2\vd}^\vw$ is a linear image of the cone of positive semidefinite matrices of appropriate size under some surjective linear map $\Lambda^*$, and similarly, its dual $(\Sigma_{n,2\vd}^\vw)^*$ is a linear pre-image of the same cone, under the adjoint map $\Lambda$}. The following Proposition makes these statements precise.
}

\begin{proposition}[\protect{\cite[Thm.~17.6]{Nesterov2000}}]\label{thm:Nesterov} Fix an ordered basis $\vq = (q_1,\dots,q_U)$ of $\mathcal{V}^\vw_{n,2\mathbf{d}}$ and an ordered basis $\vp_{i} = (p_{i,1},\dots,p_{i,L_i})$ of $\mathcal{V}_{n,d_i}$ for each $i = 1,\dots,m$. Let $\Lambda_i: \cV_{n,2\vd}^\vw \left(\equiv \rr^U\right) \to \mathbb{S}^{L_i}$ be the unique (injective) linear map satisfying $\Lambda_{i}(\mathbf{q}) = w_i\mathbf{p}_i\mathbf{p}_i^\T$, and let $\Lambda_i^*$ denote its adjoint. Then $\mathbf{s} \in \Sigma_{n,2\mathbf{d}}^\vw$ if and only if there exist matrices \mbox{$\mathbf{S}_1\succcurlyeq \mathbf{0}, \dots, \mathbf{S}_m \succcurlyeq \mathbf{0}$} satisfying 
\begin{equation}\label{eq:Nesterov-Lambdastar}
\vs = \sum_{i=1}^m\Lambda_i^*(\mathbf{S}_i).
\end{equation}
Additionally, the dual cone of $\Sigma_{n,2\mathbf{d}}^\vw$ admits the characterization 
\begin{equation}\label{eq:Nesterov-Lambda}
\left(\Sigma^\vw_{n,2\mathbf{d}}\right)^* = \left\{\vx \in \cV_{n,2\vd}^\vw \left(\equiv \rr^U\right) \ | \ \Lambda_i(\vx) \succcurlyeq \mathbf{0} ~~ \forall i=1,\dots,m\right\}.
\end{equation}
\end{proposition}
\revision{To see why $\Lambda_i$ exists and is unique, consider that each entry of the matrix of functions $w_i \vp_i \vp_i^\T$ is a polynomial of the form $w_i p_{i,j}p_{i,k}$, which by definition belongs to the space $\mathcal{V}^\vw_{n,2\mathbf{d}}$, and so it can be written uniquely as a linear combination of our chosen basis polynomials $\{q_1,\dots,q_U\}$ of this space. Thus, for any vector $\vv\in\R^U$, the $(j,k)$-th entry of the matrix $\Lambda_i(\vv)$ is the same linear combination of the components of $\vv$ which would yield $w_i p_{i,j}p_{i,k}$ if it were applied to the basis polynomials $\vq$.}

The interested reader will find a number of examples of WSOS cones $\Sigma$ and the $\Lambda$ operators representing them in different bases in \cite[Example 1]{DavisPapp2022}, using the same notation as in this paper. \revision{We briefly recall only one of them:
\begin{example}
Consider univariate polynomials of degree $4$, nonnegative on $[-1,1]$. These polynomials can be written as $\sigma_1(t) + (1-t^2)\sigma_2(t)$, where $\sigma_1\in\Sigma_{1,4}$ and $\sigma_2\in\Sigma_{1,2}$; that is, they are WSOS with the weights $w_1(t)=1$ and $w_2(t)=1-t^2$ and degree vector $\vd=(2,1)$. Representing all monomials in the monomial basis, it is well-known that $\vx = (x_0,\dots,x_4)\in(\Sigma_{n,2\vd}^\vw)^*$ if and only if
\[ \Lambda_1(\vx) := \left(\begin{smallmatrix}
	x_0 & x_1 & x_2 \\
	x_1 & x_2 & x_3 \\
	x_2 & x_3 & x_4
\end{smallmatrix}\right) \succcurlyeq 0 \text{ and }
\Lambda_2(\vx) := \left(\begin{smallmatrix}
x_0 - x_2 & x_1 - x_3 \\
x_1 - x_3 & x_2 - x_4
\end{smallmatrix}\right) \succcurlyeq 0.\]
The matrix $\Lambda_1(\vx)$ is the moment matrix, while $\Lambda_2(\vx)$ is a localizing matrix for this particular domain \cite{Laurent2009}. In the notation of Proposition \ref{thm:Nesterov}, we have $U=5$, $(L_1,L_2) = (3,2)$, and $\Lambda_i$ matrices were obtained by collecting the monomial terms in the matrices
\[
\left(\begin{smallmatrix}
	1 \\
	t \\
	t^2 
\end{smallmatrix}\right)
\left(\begin{smallmatrix}
	1 & t & t^2
\end{smallmatrix}\right)
=
\left(\begin{smallmatrix}
1 & t & t^2 \\
t & t^2 & t^3 \\
t^2 & t^3 & t^4
\end{smallmatrix}\right)  \text{ and } 
(1-t^2)\left(\begin{smallmatrix}
	1 \\
	t
\end{smallmatrix}\right)
\left(\begin{smallmatrix}
	1 & t
\end{smallmatrix}\right)
=
\left(\begin{smallmatrix}
	1-t^2 & t-t^3\\
	t-t^3 & t^2-t^4\\
\end{smallmatrix}\right).
 \]
\end{example}
}

To further lighten the notation, throughout the paper, we will assume that the weight polynomials $\vw=(w_1,\dots,w_m)$ and the degrees $\vd=(d_1,\dots,d_m)$ are fixed. We will denote the cone $\Sigma_{n,2\vd}^\vw$ by $\Sigma$ and the space of polynomials $\mathcal{V}_{n,2\mathbf{d}}^\vw$ by $\mathcal{V}$. We will usually identify the spaces $\cV$ and $\cV^*$ with $\R^U$ ($U=\dim(\cV)$), equipped with the standard inner product $\langle \vx, \vy \rangle = \vx^\T\vy$ and the induced Euclidean norm $\|\cdot\|$. For (real) square matrices, the inner product $\langle \cdot, \cdot \rangle$ denotes the Frobenius inner product. 
 
 Additionally, we use the shorthand $\Lambda$ to denote the $\R^U\to \bS^{L_1} \oplus \cdots \oplus \bS^{L_m}$ linear map $\Lambda_1(\cdot)\oplus \cdots \oplus\Lambda_m(\cdot)$ from Proposition~\ref{thm:Nesterov}. With this notation, the condition \eqref{eq:Nesterov-Lambdastar} can be written as $\vs=\Lambda^*(\vS)$ for some positive semidefinite (block diagonal) matrix $\vS\in\bS^{L_1} \oplus \cdots \oplus \bS^{L_m}$. Similarly, Eq.~\eqref{eq:Nesterov-Lambda} simplifies to
\begin{equation*}
\Sigma^* = \{\vx\in\R^U\,|\,\Lambda(\vx)\succcurlyeq \vzero\}.
\end{equation*}
The interior of this cone is simply
\begin{equation}\label{eq:intSigmaStar}
\Ssc = \{\vx\in\R^U\,|\,\Lambda(\vx)\succ \vzero\}.
\end{equation}

\subsubsection{Barrier functions and local norms in convex cones}
\revision{The theory of dual certificates builds heavily on results from the theory of barrier functions in convex optimization. Here, we introduce relevant notation, and we give a brief overview of the parts of this theory that will be needed throughout the rest of the paper.}

Let $\Lambda: \R^U \to \bS^L$ be the unique linear mapping specified in Proposition~\ref{thm:Nesterov} above, and let $\Lambda^*$ denote its adjoint. Central to our theory is the \emph{barrier function} $f: \Ssc \to \R$ defined by
\begin{equation}\label{eq:def:f}
    f(\vx) \defeq -\ln(\det(\Lambda(\vx)).
\end{equation}
Note that by Eq.~\eqref{eq:intSigmaStar}, $f$ is indeed defined on its domain. The function $f$ is twice continuously differentiable; we denote by $g(\vx)$ its gradient at $\vx$ and by $H(\vx)$ its Hessian at $\vx$. Since $f$ is strictly convex on its domain, $H(\vx)\succ \vzero$ for all $\vx \in \Ssc$ \revision{\cite[Sec.~3.1.5 and 3.2.2]{BoydVandenberghe2004}}. Consequently, we can also associate with each $\vx\in\Ssc$ the \emph{local inner product} $\langle\cdot,\cdot\rangle_\vx : \cV^* \times \cV^* \to \rr$ defined as $\langle\vy,\vz\rangle_\vx \defeq \vy^\T H(\vx) \vz$ and the \emph{local norm} $\|\cdot\|_\vx$ induced by this local inner product. Thus, $\|\vy\|_\vx = \|H(\vx)^{1/2}\vy\|$. We define the local (open) ball centered at $\vx$ with radius $r$ by $B_{\vx}(\vx,r) \defeq \{\vy\in\cV^*\,|\, \|\vy - \vx\|_{\vx} < r\}$.
Analogously, we define the \emph{dual local inner product} $\langle \cdot,\cdot \rangle_\vx^*: \cV\times\cV\to\R$ by  $\langle \vs, \vt \rangle_\vx^* \defeq \vs^\T H(\vx)^{-1}\vt$. The induced \emph{dual local norm} $\|\cdot\|_{\vx}^*$ satisfies the identity $\|\vt\|_{\vx}^* = \|H(\vx)^{-1/2}\vt\|$. 

\revision{Throughout, we will invoke several useful results concerning these norms and the barrier function $f$ in \eqref{eq:def:f}; these are enumerated in Lemma \ref{thm:f-properties}, in the Appendix of this paper. Geometrically, the key observation is that the Hessian of this barrier function, through the associated local and dual local norms, provides computable ellipsoidal neighborhoods around each point in $\Sc$ and $\Ssc$ that are contained in these cones, yielding ``safe'' bounds to round vectors in any direction without leaving the cone.}

\subsubsection{Dual certificates}\label{sec:dual-certs} 
As mentioned earlier in the introduction, our primary goal is to show the existence of a dual certificate with boundable bit size for a given WSOS polynomial. Here, we review necessary definitions and properties of dual certificates. For more extensive theory of dual certificates, see \cite{DavisPapp2022}.

\begin{definition}\label{def:dual-certificate}
Let $\vs\in\Sigma$, \revision{and denote the Hessian of the barrier function $f$ of $\Sigma^*$ defined in \eqref{eq:def:f} by $H$}. We say that the vector $\vx\in\Ssc$ is a \emph{dual certificate of $\vs$}, or simply that \emph{$\vx$ certifies $\vs$}, if $H(\vx)^{-1}\vs \in \Sigma^*$. We denote by
\[\cC(\vs) \defeq \{\vx\in\Ssc\,|\,H(\vx)^{-1}\vs \in \Sigma^*\}\]
the set of dual certificates of $\vs$. 
Conversely, for every $\vx\in\Ssc$, we denote by
\[\cP(\vx) \defeq \{\vs\in\Sigma\,|\,H(\vx)^{-1}\vs \in \Sigma^*\}\]
the set of polynomials certified by the dual vector $\vx$.
\end{definition}

\revision{
This definition is motivated by the following theorem from \cite{DavisPapp2022}, reproduced below for completeness. In words, the theorem provides an \emph{explicit closed form formula} for efficiently computing a WSOS certificate for any polynomial from its coefficient vector $\vs$ and any dual certificate $\vx$:
\begin{theorem}[\protect{\cite[Thm.~2.2]{DavisPapp2022}}]
Let $\vs\in\Ssc$ be arbitrary. Then the matrix $\vS = \vS(\vx,\vs)$ defined by
\begin{equation}\label{eq:S-def}
\vS(\vx,\vs) \defeq \Lambda(\vx)^{-1}\Lambda\left(H(\vx)^{-1}\vs\right)\Lambda(\vx)^{-1}
\end{equation}
satisfies $\Lambda^*(\vS) = \vs$. Moreover, $\vx$ is a dual certificate for $\vs \in \Sigma$ if and only if $\vS\succcurlyeq 0$, which in turn is equivalent to $H(\vx)^{-1}\vs\in\Sigma^*$. 
\end{theorem}

Note that as long as $\Lambda$ maps rational vectors to rational matrices (which is the case, for instance, when polynomials are represented in commonly used bases such as the standard monomial basis or the Chebyshev basis), then $\vS$ is a rational matrix for every rational coefficient vector $\vs$.
}

It is immediate from Definition \ref{def:dual-certificate} that if $\vx$ is a dual certificate of the polynomial $\vs$, then every positive multiple of $\vx$ is also a dual certificate for every positive multiple of $\vs$. \revision{Crucially, the same is true for \emph{small perturbations} of $\vx$ and $\vs$; see Proposition \ref{thm:x-ycert} below.}

From Lemma~\ref{thm:f-properties} (claim \ref{item:bijection}) in the Appendix, we know that for every $\vs\in\Sigma^\circ$ there exists a unique $\vx \in \Ssc$ satisfying $\vs=-g(\vx)$. This vector is a dual certificate of $\vs$, since
\[H(\vx)^{-1}\vs = -H(\vx)^{-1}g(\vx) \overset{\eqref{eq:gH-identities}}{=} \vx \in \Ssc.\]
Thus, every polynomial in the interior of the WSOS cone $\Sigma$ has a dual certificate.
\begin{definition}
When $-g(\vx) = \vs\, (\in\!\Sigma^\circ)$, we say that $\vx$ is the \emph{gradient certificate of $\vs$}.
\end{definition}
\revision{Simple calculus reveals the closed-form formula for the negative gradient: $-g(\vx) = \Lambda^*(\Lambda(\vx)^{-1})$; see also Lemma \ref{thm:f-properties} in the Appendix. However, since $\Lambda^*$ is in general not injective, the nonlinear system $\vs = \Lambda^*(\Lambda(\vx)^{-1})$ cannot be solved for $\vx$ in closed form; only the $\vx\to\vs$ map is easily computable, not the converse. The same mapping $-g$ has also been recently studied by Lasserre \cite{Lasserre2022} and others \cite{DeCastro-etal2021}. We shall elaborate more on this connection in Section \ref{sec:discussion}.
}

The following proposition gives two sufficient, although not necessary, conditions for $\vx \in \Sigma^*$ to certify a polynomial $\vt$. \revision{It also reveals that $\cC(\vs)$ and $\cP(\vx)$ are \emph{full-dimensional} cones, that is, they have a non-empty interior: every sufficiently small perturbation of the gradient certificate of $\vs$ certifies every sufficiently small perturbation of $\vs$.}
\begin{proposition}[\protect{\cite[Theorem 2.4 and Corollary 2.5]{DavisPapp2022}}]\label{thm:x-ycert}
Suppose that $\vx\in\Sigma^*$ and $\vs=-g(\vx)$. \begin{enumerate} 
\item Then $\vx$ is a dual certificate for every polynomial $\vt$ satisfying $\|\vt - \vs\|_{\vx}^* \leq 1$. 
\item If $\vy$ is a vector that satisfies the inequality $\|\vx - \vy\|_\vx < \frac{1}{2}$, then $\vy\in\Sigma^*$, and $\vx$ certifies $\vt=-g(\vy)$.
\end{enumerate}
\end{proposition}

\revision{Two very detailed examples illustrating the concept of dual certificates, the gradient certificate, and the construction of explicit WSOS representations from dual certificates can be found in our previous work \cite[Examples 2 and 3]{DavisPapp2022}.}

\subsubsection{Bit sizes of certificate vectors}
Recall that the bit size of an integer $y\in\mathbb{Z}$ is defined as $1+\lceil \log_2(|y|+1) \rceil$, and that the bit size of a vector $\vy\in\mathbb{Z}^n$ can be bounded from above by $n$ times the bit size of its the largest (in size) component. As we are interested in the orders of magnitude of bit sizes of dual certificates (e.g., whether they are linear or polynomial or exponential functions of parameters such as the degree or the number of variables of the certified polynomials), it will be convenient but equally informative to substitute this quantity with the simpler $\log(\|\vy\|_\infty)$.

\section{Rational certificates with boundable bit bize}\label{sec:gradbounds} 
The goal of this section is to bound the norm of an \emph{integer} dual certificate $\overline{\vy} \in \Sigma^*$ of a polynomial $\vt \in \Sc$. We consider different bounds, some of which depend only on the number of variables $n$, the degree $d$, and $\vt$, and others that are expressed in terms of other computable or interpretable parameters introduced later in this section.

The strategy to derive these bounds is as follows. In Section~\ref{sec:bounding-denominators}, we show that dual certificates suitably close to the gradient certificate can be rounded to nearby rational dual certificates with small denominators. Then, in Section~\ref{sec:bounding-norms}, we show that these certificates also have small norms. These two results add up to Theorem~\ref{thm:bitsize-grad} bounding the bit size of an integer dual certificate.



\subsection{Hessian bounds}\label{sec:bounding-denominators}
Recall from Proposition \ref{thm:x-ycert}, Statement 2, that if $\vx \in \Sigma^*$ and  $\|\vx - \vy\|_{\vx} < \frac{1}{2}$, then $\vx$ certifies $\vt = -g(\vy)$ to be WSOS. This certificate $\vx$ need not be a  rational vector, let alone a vector with small denominator. However, Lemma \ref{thm:rounding} below guarantees the existence of a nearby rational certificate $\vx_N$ for $\vt$ with boundable denominators. In this Lemma, and throughout the rest of the section, we shall continue using $g$ and $H$ to denote the gradient and Hessian of the function $f$ defined in \eqref{eq:def:f}.

\begin{lemma}\label{thm:rounding} Let $\vt \in \Sc$ be the coefficient vector of a polynomial whose gradient certificate is $\vy \in \Ssc$. Let $0<r_1<r_2$ be arbitrary, and suppose that $\vx \in \Sigma^*$ satisfies $\|\vx - \vy\|_{\vx} \leq r_1 < 1/2$. Let $U=\dim(\Sigma)$, and choose an integer denominator $N > 0$ to satisfy 
\begin{equation}\label{eq:roundingthm}
\|H(\vx)^{1/2}\| \leq \frac{2N}{\sqrt{U}} \left(\frac{r_2 - r_1}{1 + r_2}\right).
\end{equation}
Then every $\vx_N \in \frac{1}{N} \zz^U$ with $\|\vx_N - \vx\|_2 \leq \frac{\sqrt{U}}{2N}$ satisfies $\|\vx_N - \vy\|_{\vx_N} \leq r_2$. In particular, if $r_2<1/2$, then rounding $\vx$ componentwise to the nearest vector in $\frac{1}{N} \zz^U$ results in a rational dual certificate of $\vt$.  
\end{lemma} 

\begin{proof} 
By self-concordance (inequality \eqref{eq:self-concordance} in Lemma \ref{thm:f-properties}), we have 
\begin{align*}
    \|\vx_N - \vy\|_{\vx_N} &\overset{\eqref{eq:self-concordance}}{\leq} \frac{\|\vx_N - \vy\|_{\vx}}{1-\|\vx_N - \vx\|_\vx} \\
    &\leq \frac{\|\vx_N - \vy\|_\vx}{1 - \frac{\sqrt{U}}{2N}\|H(\vx)^{1/2}\|}\\
    &\overset{\eqref{eq:roundingthm}}{\leq} \frac{1+r_2}{1+r_1}\|\vx_N - \vy\|_{\vx} \\
    &\leq  \frac{1+r_2}{1+r_1}\left( \|\vx_N - \vx\|_\vx + \|\vx - \vy\|_{\vx} \right)\\
    &\leq  \frac{1+r_2}{1+r_1}\left( \frac{\sqrt{U}}{2N}\|H(\vx)^{1/2}\| + r_1 \right) \\
    &\overset{\eqref{eq:roundingthm}}{\leq}  \deletethis{\frac{1+r_2}{1+r_1}\left( \frac{\sqrt{U}}{2N}\frac{2N}{\sqrt{U}}\left( \frac{r_2 - r_1}{1+r_2}\right) + r_1 \right) \\
    &=}r_2,
\end{align*}
which proves the first part of the claim. For the second part, if $\vx_N$ is the result of component-wise rounding $\vx$ to the nearest vector in $\frac{1}{N}\zz^U$, then $\|\vx_N - \vx\|_{\infty} \leq \frac{1}{2N}$, so $\|\vx_N - \vx\|_{2} \leq \frac{\sqrt{U}}{2N}$. Then the first part of the claim shows that $\vx_N$ is a certificate for $\vt$. 
\end{proof} 

In the corollary below, we consider the particular case in which the known certificate $\vx$ for $\vt$ is the gradient certificate (i.e., $\vx=\vy$).
\deletethis{ We can use Lemma \ref{thm:rounding} to construct a more explicit bound on the denominator size needed for a certificate obtained from rounding the gradient certificate of a polynomial, as made explicit in the next Corollary. Recall from \cite{DavisPapp2022} that every weighted sum-of-squares polynomial $\vt \in \Sigma^\circ$ has a gradient certificate $\vy$. However, the gradient certificate $\vy$ is not guaranteed to be rational. Lemma \ref{thm:rounding} below guarantees the existence of a rational certificate $\vy_N$ for $\vt$ with boundable denominators. }

\begin{corollary}\label{thm:round-grad} Let $\vy \in \Ssc$ be the gradient certificate for $\vt \in \Sc$. Suppose $N$ satisfies
\begin{equation*}
\frac{3\sqrt{U}\|H(\vy)^{1/2}\|}{2} \leq N,
\end{equation*}
and suppose that $\vy_N$ satisfies $\|\vy - \vy_N\|_2 \leq \frac{\sqrt{U}}{2N}$. Then $\vy_N$ certifies $\vt$.
%
\end{corollary} 

\begin{proof} Substituting $r_1=0$ and $r_2 = 1/2$ into Lemma \ref{thm:rounding} yields the claim.
\end{proof} 

The denominators in Lemma \ref{thm:rounding} and Corollary \ref{thm:round-grad} depend on the norm $\|H(\vy)^{1/2}\|$, which is computable for a known $\vy$. If $\vy$ is not known explicitly, we can use the following upper bound for $\|H(\vy)^{1/2}\|$, which depends only on the norm of the coefficient vector $\vt$ of the polynomial $t(\cdot)$, the set $S_\vw$ defined in \eqref{eq:Swdef}, and the chosen basis $\vq$ of $\Sigma$. 

\begin{lemma}\label{thm:hess} 
 Let $\{\vz_1,\dots,\vz_s\} \subseteq S_\vw$ be a unisolvent set for $\Sigma$, with $s\geq U = \dim(\Sigma)$. Let $\vt \in \Sigma^\circ$, and let $\vy \in \Ssc$ satisfy $-g(\vy) = \vt$. Choose $\alpha_1,\dots,\alpha_s > 0$, and define the matrix $\mathbf{M}$ as \begin{equation}\label{eq:mdef}
\mathbf{M} \defeq  \sum_{i=1}^s\alpha_i\vq(\vz_i)\vq(\vz_i)^\T.
\end{equation}
Then $\vM$ is a positive definite matrix, and
\[ \|H(\vy)\|_2 \leq \cond(\vM) \|\vt\|_2^2,\]
where $\cond(\vM)= \frac{\lambda_{\max}(\vM)}{\lambda_{\min}(\vM)}$ is the condition number of $\vM$. 
\end{lemma} 
\begin{proof}
We begin by showing that $\vM$ is positive definite. Let $\vv$ be a unit-norm eigenvector of the smallest eigenvalue of $\vM$, and consider the polynomial $v(\cdot):=\vv^\T\vq(\cdot)$. Then 
\[ 
\lambda_{\min} \left(\mathbf{M} \right) = \vv^\T\Big(\sum_{i=1}^s\alpha_i\vq(\vz_i)\vq(\vz_i)^\T \Big)\vv = \sum_{i=1}^s \alpha_iv(\vz_i)^2. 
\]
Since $\{\vz_1,\dots,\vz_s\}$ is unisolvent and $\vv\neq\vzero$, it follows that $v$ is not the constant $0$ polynomial, moreover $v(\vz_i) \neq 0$ for at least one $i$. Hence $\sum_{i=1}^U v(\vz_i)^2 > 0$. As each $\alpha_i$ is positive, it follows that $\lambda_{\min}(\mathbf{M}) > 0$.


Now, we proceed with the proof of the claimed inequality.  Recall from Section \ref{sec:preliminaries} that $\Lambda_i(\vq(\cdot)) = w_i(\cdot)\vp(\cdot)\vp(\cdot)^\T$ for each $i = 1,...,m$. Thus, 
\begin{equation}\label{eq:t-hess-1}
\begin{split}
\sum_{i=1}^mw_i(\cdot)\vp_i(\cdot)^\T\Lambda_i(\vy)^{-1}\vp_i(\cdot) &= \sum_{i=1}^m\text{tr}(\Lambda_i(\vq(\cdot))\Lambda_i(\vy)^{-1})\\
&\overset{\eqref{eq:g}}{=} \vq(\cdot)^\T(-g(\vy)) = \vq(\cdot)^\T\vt = t(\cdot). 
\end{split}
\end{equation}
Thus, letting  $\Lambda(\cdot)$ represent $ \Lambda_1(\cdot) \oplus  \dots \oplus \Lambda_m(\cdot)$, we have for all $\vz \in S_\vw$, 
\begin{equation}\label{eq:hesseq1}
\begin{split}
\vq(\vz)^\T H(\vy)\vq(\vz) &\overset{\eqref{eq:H}}{=} \text{tr}(\Lambda(\vq(\vz))\Lambda(\vy)^{-1}\Lambda(\vq(\vz))\Lambda(\vy)^{-1}) \\
&=\sum_{i=1}^m\text{tr}\left(w_i(\vz)\vp_i(\vz)\vp_i(\vz)^\T\Lambda_i(\vy)^{-1}w_i(\vz)\vp_i(\vz)\vp_i(\vz)^\T\Lambda_i(\vy)^{-1}\right)\\
&=\sum_{i=1}^m  \text{tr}\left(w_i(\vz)\vp_i(\vz)^\T\Lambda_i(\vy)^{-1}\vp_i(\vz)w_i(\vz)\vp_i(\vz)^\T\Lambda_i(\vy)^{-1}\vp(\vz) \right) \\
&=\sum_{i=1}^m\left(w_i(\vz)\vp_i(\vz)^\T\Lambda_i(\vy)^{-1}\vp_i(\vz)\right)^2\ \\
&\overset{(*)}{\leq}\sum_{i=1}^m\sum_{j=1}^m\left(w_i(\vz)\vp_i(\vz)^\T\Lambda_i(\vy)^{-1}\vp_i(\vz)\right)\left(w_j(\vz)\vp_j(\vz)^\T\Lambda_j(\vy)^{-1}\vp_j(\vz)\right)\\
&\overset{\eqref{eq:t-hess-1}}{=} t(\vz)^2,
\end{split}
\end{equation}
with the inequality in $(*)$ due to the facts that $w_i(\vz) \geq 0$ whenever $\vz \in S_\vw$ and that $\Lambda_i(\vy)^{-1}$ is positive definite. 

Let $\{\vz_1,\dots,\vz_s\}$ be the unisolvent point set in the definition of $\mathbf{M}$ above. Then we have
\begin{equation*}
\begin{split}
\|H(\vy)\|_2 \leq \text{tr}(H(\vy)) = \langle H(\vy), \mathbf{I} \rangle &\leq \left\langle H(\vy), \frac{\mathbf{M}}{\lambda_{\min}(\mathbf{M})}  \right\rangle \\
&\overset{\eqref{eq:hesseq1}}{\leq} \frac{\sum_{i=1}^s \alpha_i t(\vz_i)^2}{\lambda_{\min}\left(\mathbf{M}\right)} \\
&= \frac{\vt^\T \mathbf{M}\vt}{\lambda_{\min}\left(\mathbf{M}\right)} \\
&\leq \frac{\|\vt\|_2^2\lambda_{\max}\left( \mathbf{M}\right)}{\lambda_{\min}\left(\mathbf{M}\right)} \\
&=\text{cond}(\mathbf{M})\|\vt\|_2^2. \qedhere
\end{split}
\end{equation*}
\end{proof} 

We can use the bound in Lemma \ref{thm:hess} to bound the denominators needed in Lemma \ref{thm:rounding} and Corollary \ref{thm:round-grad}. 

\begin{theorem}\label{thm:grad-denominators} Let $\vy$ be the gradient certificate for $\vt$. Let $U = \dim\left(\Sigma\right)$, and let $\mathbf{M}$ be defined as in \eqref{eq:mdef}.  Let 

\[
 N = \left\lceil\frac{3}{2}\sqrt{U\cond(\mathbf{M})}\|\vt\|_2\right\rceil.
\]
Then every $\vy_N \in \frac{1}{N}\zz^U$ with $\|\vy_N - \vy\|_{2} \leq \frac{\sqrt{U}}{2N}$ is a certificate for $\vt$.
\end{theorem} 

\begin{proof} From Lemma \ref{thm:hess}, we have $\|H(\vy)^{1/2}\|_{2} \leq \sqrt{\text{cond}(\mathbf{M})} \|\vt\|_{2}$. Therefore, $N$ satisfies
\[
\frac{3\sqrt{U}}{2}\|H(\vy)^{1/2}\| \overset{\text{Lem.} \ref{thm:hess}}{\leq} \frac{3\sqrt{U}}{2} \sqrt{\text{cond}(\mathbf{M})} \|\vt\|_{2}
\leq N. 
\]
Then by Corollary \ref{thm:round-grad}, every $\vy_N \in \frac{1}{N}\zz^U$ with $\|\vy_N - \vx\|_{2} \leq \frac{\sqrt{U}}{2N}$ certifies $\vt$. 
\end{proof} 

Two remarks are in order. First, in order to obtain the smallest possible upper bound on the denominator $N$ that works, the goal should be to minimize the $\cond(\vM)$ with respect to the points $\vz_i$ in the definition of $\vM$---a likely impossible task in general. However, any unisolvent set from $S_\vw$ provides a bound, and that is generally a relatively straightforward task to find. Second, the value $\sqrt{U\cond(\vM)}$ is a property of the cone $\Sigma$, independent of $\vt$, therefore this optimization (or selection of suitable points $\vz_i$) needs to be performed only once for a given cone $\Sigma$. In Section \ref{sec:gradcerts-bases}, we shall use natural candidate points for interesting special cases, for which the $\cond(\vM)$ is computable in closed form.

\subsection{Bounding certificate norms}\label{sec:bounding-norms}
Now, we turn our attention to bounding the norms of rational certificates for a given polynomial. The results make use of two new parameters of the cone $\Sigma$ and its representation via the operator $\Lambda$. The first one is the \emph{barrier parameter} of the barrier function $f$ defined (in the notation of Proposition~\ref{thm:Nesterov}) as
\[ \nu \defeq \sum_{i=1}^m L_i \]
(see also Lemma~\ref{thm:f-properties}, Statement \ref{item:loghom}). The other is the constant $k_1^\ve$ defined in our next Lemma. This statement is analogous to Equation (3.8) in the proof of \cite[Theorem 3.5]{DavisPapp2022}, but is included here for completeness.

\begin{lemma}\label{thm:ynorm}
Let $\vy \in \left(\Sigma^*\right)^\circ$ be the gradient certificate for $\vt \in \Sigma^\circ$, and let $\ve \in \Sc$. Let $\varepsilon > 0$, and suppose that $\vt - \varepsilon\ve \in \bd(\Sigma)$. Define
\begin{equation}\label{eq:k1def}
k_1^\ve \defeq \min\left\{\ve^\T\vv \ | \ \vv \in \Sigma^*, \|\vv\|_{\infty} = 1\right\}. 
\end{equation}
Then 
\[
\|\vy\|_{\infty} \leq \frac{\nu}{k_1^\ve\varepsilon}.
\]
\end{lemma} 

\begin{proof}
Observe that the minimum in the definition \eqref{eq:k1def} exists (as $\Sigma^*$ is a closed and non-trivial cone) and $k_1^\ve > 0$ because $\ve \in \Sigma^\circ$. We now have
\begin{align*}
    \nu &\overset{\eqref{eq:gH-identities}}{=}\left  \langle -g\left(\frac{\vy}{\|\vy\|_\infty}\right), \frac{\vy}{\|\vy\|_\infty}\right \rangle \\
    &\overset{\eqref{eq:log-homogeneity}}{=} \|\vy\|_\infty \left\langle \vt, \frac{\vy}{\|\vy\|_\infty} \right\rangle \\
    &= \|\vy\|_\infty\left(\left\langle \vt - \varepsilon\ve, \frac{\vy}{\|\vy\|_\infty} \right\rangle + \varepsilon \left\langle \ve, \frac{\vy}{\|\vy\|_\infty}\right\rangle  \right) \\
    &\geq0 + \|\vy\|_\infty\varepsilon k_1^\ve,
\end{align*}
and the claimed upper bound on $\|\vy\|_\infty$ follows.
\end{proof}

\begin{remark}
Note that this bound can be computed efficiently using numerical methods for any given $\vt$ and chosen $\ve$. First, the $\varepsilon$ corresponding to $\vt$ and $\ve$ can be approximated (or bounded below) by simple line search; the bottleneck is testing membership in $\Sigma$ in each step. Second, although the minimization problem \eqref{eq:k1def} is not convex, its (global) optimal value can be computed by solving $2U$ efficiently solvable convex optimization problems, since
\begin{multline}\label{eq:k1e-convex}
k_1^\ve = \min_{1\leq i \leq U} \big(\inf \{ \ve^\T \vv \ | \ \vv \in \Sigma^*, \vv_{i} = 1, \|\vv\|_\infty \leq 1\}, \\\inf \{ \ve^\T \vv \ | \ \vv \in \Sigma^*, \vv_{i} = -1, \|\vv\|_\infty \leq 1\}\big).
\end{multline}
This reformulation will also allow us to bound $k_1^\ve$ from below using convex programming duality in the next section.
\end{remark}

Since we would like the tightest possible upper bound on $\|\vy\|_{\infty}$, we seek an $\ve \in \Sc$ which maximizes the quantity $k_1^\ve \varepsilon$. Lemma~\ref{thm:optimal-e} below illustrates that, for a fixed $\vt \in \Sigma^\circ$, choosing $\ve = \vt$ (with the resulting $\varepsilon=1$) is the optimal choice. Nevertheless, the above, more general, form of Lemma \ref{thm:ynorm} is useful when we are concerned with bounding certificates for families of polynomials (such as all univariate polynomials of degree $d$) in terms of interpretable parameters such as the number of variables or the degree of the polynomials. In this context, it is often more convenient to use the bound selecting any (convenient) $\ve$ in the quantity $k_1^\ve \varepsilon$, instead of $k_1^\vt$.


\begin{lemma} \label{thm:optimal-e} Let $\vy \in \Ssc$ be the gradient certificate for $\vt \in \Sigma^\circ$. Then for all $\ve \in \Sc$ and $\varepsilon > 0$ such that $\vt - \varepsilon\ve \in \bd(\Sigma)$, we have $k_1^\vt \geq k_1^\ve \varepsilon$.
\end{lemma}

\begin{proof}
Denote by $E$ the set 
\[
E = \{\ve \in \Sc \ | \ \text{there exists } \varepsilon > 0 \text{ such that } \vt - \varepsilon \ve \in \text{bd}(\Sigma) \},
\]
and let $\hat{E}$ be the set 
\[
\hat{E} = \{\ve \in \Sc \;|\;\vt -  \ve \in \text{bd}(\Sigma) \}
 = \{\vt - \vb \;|\;\vb \in \bd(\Sigma)\}.\] Let $V$ represent the set 
\[
V = \{ \vv \in \Sigma^* \ | \ \|\vv\|_{\infty} = 1 \},
\] and let $\varepsilon(\vt, \ve)$ be the largest $\varepsilon$ such that $\vt - \varepsilon \ve$ is on the boundary of $\Sigma$, for given $\vt$ and $\ve$. 
With this notation, we may now write
\begin{align*}
\max_{\ve \in \rr^U} k_1^\ve \varepsilon &= \max_{\ve \in E} \Big( \min_{\vv \in V} \varepsilon (\vt, \ve) ~ \ve^\T \vv \Big)\\
&= \max_{\ve \in \hat{E}} \Big( \min_{\vv \in V}  ~ \ve^\T \vv \Big) \\
&= \max_{\vb \in \text{bd}(\Sigma)} \Big( \min_{\vv \in V}  ~ (\vt - \vb)^\T \vv \Big) \\
&\leq \min_{\vv \in V} \Big(\max_{\vb \in \text{bd}(\Sigma)} \vt^\T\vv - \vb^\T\vv \Big)\\
&= \min_{\vv \in V} \Big(\vt^\T\vv + \max_{\vb \in \bd(\Sigma)} (-\vb^\T\vv)\Big)\\
&\leq \min_{\vv \in V} \vt^\T\vv \\
&= k_1^\vt, 
\end{align*}
where the first inequality comes from the weak duality theorem of convex optimization, and the second one follows from $\vv\in\Sigma^*$ and $\vb\in\Sigma$. 
\end{proof}



Having bounded the norm of the gradient certificate in Lemma \ref{thm:ynorm}, we can now bound the norm of a rounded gradient certificate.

\begin{lemma} \label{thm:roundedynorm}
Let $\vy \in \left(\Sigma^*\right)^\circ$ be the gradient certificate for a polynomial with coefficient vector $\vt \in \Sigma^\circ$.  Let $U = \dim\left(\Sigma\right)$, $\ve \in \Sc$, and $\varepsilon > 0$, and suppose $\vt - \varepsilon\ve\in\bd(\Sigma)$. Denote by $\nu$ the barrier parameter for $\Lambda$ and $\mathbf{M}$ the matrix defined as in \eqref{eq:mdef}. Let $N =  \left\lceil\frac{3}{2}\sqrt{U\cond(\mathbf{M})}\|\vt\|_{2}\right\rceil$, and suppose $\vy_N \in \rr^U$ satisfies $\|\vy_N - \vy\|_{\infty} \leq \frac{1}{2N}$. Then $\vy_N$ is a certificate for $\vt$ with
\[
\|\vy_N\|_\infty \leq \frac{1}{3 \sqrt{U\cond(\mathbf{M})}\|\vt\|_{2}} + \frac{\nu}{k_1^\ve\varepsilon}.
\]
\end{lemma}
\begin{proof} The fact that $\vy_N$ is a certificate for $\vt$ has already been shown in Lemma  \ref{thm:grad-denominators}. 

Recall from Lemma \ref{thm:ynorm} that $\|\vy\|_\infty \leq \frac{\nu}{k_1^\ve\varepsilon}$. Moreover, by our choice of $\vy_N$ and $N$, we know  
\begin{equation}\label{eq:ynlemma1}
    \|\vy_N - \vy\|_\infty \leq \frac{1}{2N} \leq
    \frac{1}{3 \sqrt{U\text{cond}(\mathbf{M})}\|\vt\|_{2}}.
\end{equation}
Therefore, by the triangle inequality, we have
\begin{equation*}
\|\vy_N\|_\infty \leq \|\vy_N - \vy\|_\infty + \|\vy\|_\infty \overset{\eqref{eq:ynlemma1}, \text{ Lem. } \ref{thm:ynorm}}{\leq} \frac{1}{3 \sqrt{U\text{cond}(\mathbf{M})}\|\vt\|_{2}} + \frac{\nu}{k_1^\ve\varepsilon}. \qedhere
\end{equation*}
\end{proof} 

\subsection{Bounds on integer certificate norms}
Compiling the results from this section, we are now prepared to state a result bounding the largest magnitude component of an integer certificate for a polynomial $\vt \in \Sigma^\circ$. 

\begin{theorem}\label{thm:bitsize-grad} Let $U = \dim\left(\Sigma\right)$, and let $\mathbf{M}$ be defined as in \eqref{eq:mdef}. Let $\ve \in \Sc$, and let $k_1^\ve$ be defined as in \eqref{eq:k1def}. Let $\nu$ be the barrier parameter of $\Lambda$, and let $\vt \in \Sigma^\circ$ with $\vt - \varepsilon\ve$ on the boundary of $\Sigma$. Then there exists an integer certificate $\overline{\vy}$ for $\vt$ with
\deletethis{
\begin{align*}
\|\overline{\vy}\|_\infty &\leq \frac{1}{2} +  \left\lceil\frac{3}{2}U\sqrt{\text{cond}(\mathbf{M})}\|\vt\|_{\infty}\right\rceil\left(\frac{\nu}{k_1^\ve\varepsilon}\right) \\
&\leq \mathcal{C}\left(\frac{U\sqrt{\text{cond}(\mathbf{M})}\|\vt\|_{\infty}\nu}{k_1^\ve\varepsilon} \right).
\end{align*}}
\begin{align*}
\|\overline{\vy}\|_\infty &\leq \frac{1}{2} +  \left\lceil\frac{3}{2}\sqrt{U\cond(\mathbf{M})}\|\vt\|_{2}\right\rceil\left(\frac{\nu}{k_1^\ve\varepsilon}\right). 
\end{align*}
\end{theorem}

\begin{proof} \deletethis{ Let $\vt \in \Sigma$, and let $\vy$ be the gradient certificate for the  $\vt$. Let $\vy_N$ be the result of rounding $\vy$ componentwise to the nearest vector in $\frac{1}{2N}\zz^U$,  with  
\begin{equation}\label{eq:bitsize-grad-1}
N = \left\lceil\frac{3}{2}\sqrt{U\text{cond}(\mathbf{M})}\|\vt\|_{2}\right\rceil. 
\end{equation}From Corollary \ref{thm:grad-denominators}, we know that $\vy_N$  is a certificate for $\vt$. 
Since $\vy$ is rounded componentwise to $\vy_N$, we also have that $\|\vy_N - \vy\|_{\infty} \leq \frac{1}{2N}$. Thus, from Lemma \ref{thm:roundedynorm}, we have \begin{equation}\label{eq:bitsize-grad-2}
\|\vy_N\|_\infty \leq \frac{1}{3 \sqrt{U\text{cond}(\mathbf{M})}\|\vt\|_{2}} + \frac{\nu}{k_1^\ve\varepsilon}.
\end{equation}
}
Recall from Sec. \ref{sec:dual-certs} that any positive multiple of $\vy_N$ is also a certificate for $\vt$. Hence, the integer vector $\overline{\vy} \defeq N\vy_N$ is a certificate for $\vt$. Using Lemma~\ref{thm:roundedynorm} and Theorem~\ref{thm:grad-denominators}, we have
\begin{align*}
\|\overline{\vy}\|_{\infty} = N \|\vy_N\|_{\infty} &\leq \left\lceil\frac{3}{2}\sqrt{U\text{cond}(\mathbf{M})}\|\vt\|_{2}\right\rceil\left(\frac{1}{3 \sqrt{U\text{cond}(\mathbf{M})}\|\vt\|_{2}} + \frac{\nu}{k_1^\ve\varepsilon}\right) \\
&=\frac{1}{2} +  \left\lceil\frac{3}{2}\sqrt{U\cond(\mathbf{M})}\|\vt\|_{2}\right\rceil\left(\frac{\nu}{k_1^\ve\varepsilon}\right). \qedhere
\end{align*}
\end{proof}

\section{Bit size bounds for rational certificates in particular bases}\label{sec:gradcerts-bases}

In this section we refine the result of Theorem \ref{thm:bitsize-grad} in a few well-studied and computationally relevant special cases such as the cones of univariate polynomials nonnegative on the real line or on a bounded interval. These results complement existing ones on the bit sizes of conventional sum-of-squares certificates of nonnegative univariate polynomials, such as those summarized in the Introduction. We emphasize that our approach yields an efficiently computable bound for a variety of WSOS cones even in the multivariate case. We consider one of these in Section \ref{sec:multi-interp}. We also consider different choices of bases, which are relevant for practical computation, specifically polynomials represented in the Chebyshev basis and polynomials represented as interpolants.

The results presented in this section come in two flavors, motivated by two different mindsets and two different families of applications in which nonnegativity certification is important. From the perspective of \emph{optimization}, the fundamental task is to certify a bound as close to the global minimum of a polynomial (on $S_\vw$) as possible, and therefore one is inherently concerned with certifying polynomials close to the boundary of $\Sigma$. In this setting, $\varepsilon$ (specifically with the choice of $\ve=\vone$) is arguably the most important parameter in Theorem \ref{thm:bitsize-grad}, even for a fixed number of unknowns $n$ and fixed degree $d$, and one of the most pertinent questions is the dependence of the bit size of the certificates on $\varepsilon$, as $\varepsilon$ tends to $0$. The common simplifying assumption that the coefficient vector $\vt$ is integer is not particularly convenient or necessary; $\vt$ can be any rational vector. On the other hand, from the perspective of \emph{theoretical computer science} and applications such as automated theorem proving, the fundamental task is to certify that a given polynomial is nonnegative on a given $S_\vw$. Although the dependence of the complexity of the certificate on $\varepsilon$ (with any $\ve$) is still informative, the primary concern is the asymptotic complexity of the certificate as the input size increases. It is convenient to assume that $\vt$ is an integer vector, and the relevant question is the bit size of the certificates as a function of $n$, $d$, and $\tau=\log(\|\vt\|_\infty)$. Therefore, we consider both bounds involving $\varepsilon$ (for general $\vt$) and bounds that are strictly functions of $(n,d,\tau)$, assuming that $\vt$ is an integer vector.

In each case, we will use the bound given in Theorem \ref{thm:bitsize-grad} as a starting point. Since $U$ and $\nu$ can be expressed in terms of $d$ and $n$, a result  depending only on $d$, $n$, and $\tau$, and possibly $\varepsilon$, requires a lower bound on $k_1^\ve$ or $k_1^\ve \varepsilon$ (for some $\ve$) and an upper bound on $\text{cond}(\mathbf{M})$, for some $\mathbf{M}$ in the form given in \eqref{eq:mdef}, in each case. 

\subsection{Univariate polynomials over the real line}
We first consider the most well-studied special case, univariate polynomials nonnegative over the real line, which coincide with univariate SOS polynomials, represented in the monomial basis. In this case, $\Lambda(\vx)$ is the standard moment matrix (positive definite Hankel matrix) corresponding to the pseudo-moment vector $\vx\in\Sc$.

\begin{theorem}\label{thm:monomial-bit-eps-condensed}
Suppose that $\Sigma=\Sigma_{1,2d}$ (univariate sum-of-squares polynomials of degree $2d$) and that we represent all polynomials in the monomial basis---that is, in the notation of Proposition \ref{thm:Nesterov}, the ordered bases $\vp$ and $\vq$ are the standard monomial bases of degree $d$ and $2d$, respectively. Then the following hold:
\begin{enumerate} 
    \item For $\ve = \Lambda^*(\mathbf{I})$, the coefficient vector of the sum of monomial squares polynomial $z \mapsto 1 + z^2 + \dots + z^{2d}$, we have $k_1^\ve \geq 1$. 
    \item There exists a matrix $\mathbf{M}$ of the form given in \eqref{eq:mdef} such that $\cond(\vM) \leq 3.21^{2d+1}/2$.
    \item For every $\vt \in \Sigma^\circ$ with $\vt - \varepsilon \ve \in \bd(\Sigma)$, there exists an integer dual certificate $\overline{\vy} \in \Sigma^\circ \cap \zz^{2d+1}$ for $\vt$ with
\begin{equation*}
\|\overline{\vy}\|_{\infty}
\leq \frac{1}{2} + 4d^{3/2}3.21^{d+1/2}\frac{\lceil \|\vt\|_2\rceil }{\varepsilon},
\end{equation*}
whose largest component has bit size bounded as
\begin{align*}
    \log(\|\overline{\vy}\|_{\infty}) \approx \mathcal{O}\left(d + \log(\|\vt\|_2) + \log(1/\varepsilon) \right).
\end{align*}

\end{enumerate} 
\end{theorem} 

\begin{proof} \phantom{of the opera} 
\begin{enumerate} 
\item 
First, observe that since $\mathbf{I}$ is positive definite, $\ve \in \Sigma^\circ$. Now, we have 
\[
\ve^\T\vv = \langle \Lambda^*(\mathbf{I}), \vv \rangle = \langle \mathbf{I}, \Lambda(\vv) \rangle = \text{tr}(\Lambda(\vv)), 
\]
Combining this equation with \eqref{eq:k1def}, we get
\begin{equation}\label{eq:k1e-mono}
k_1^\ve = \min\{\tr(\Lambda(\vv)) \ | \ \Lambda(\vv) \succcurlyeq \vzero, \|\vv\|_{\infty} = 1\}. 
\end{equation}
It is well known that in this setting (nonnegative univariate polynomials represented in the monomial basis), $\Lambda(\cdot)$ maps a vector $\vv$ to its corresponding Hankel matrix, whose diagonal consists of components of $\vv$. (See, e.g., \cite[Example 1]{DavisPapp2022}.) Because the matrix $\Lambda(\vv)$ is positive semidefinite, its largest element (in absolute value) is a nonnegative component of $\vv$ on the diagonal of $\Lambda(\vv)$. Hence, in \eqref{eq:k1e-mono} we have  $\text{tr}(\Lambda(\vv)) \geq \|\vv\|_{\infty} = 1$. Therefore $k_1^\ve \geq 1$. 

\item When $\vq(z)=(1,z,\dots,z^{2d})$, the matrix $\vM$ defined in \eqref{eq:mdef} is a positive semidefinite Hankel matrix of order $(2d+1)$. It is well known that a partial converse also holds, and every positive definite Hankel matrix of order $(2d+1)$ can be written in this form, because positive definite matrices can be identified with truncated moment matrices of Borel measures $\mu$ supported on the real line (see, e.g., \cite[Theorem 3.146]{BlekhermanParriloThomas2013}), and via Gaussian quadrature \cite[Chap.~7, Thm. 3]{Krylov1959} this measure $\mu$ can also be chosen to be a discrete one supported on at most $2d+1$ points.

Bounds on the condition numbers of positive definite Hankel matrices have been studied by many authors. Our Statement 2 follows from a bound of Beckermann \cite[Thm. 3.6]{Beckermann2000}, which states that the positive definite $n\times n$ Hankel matrix of the lowest condition number has condition number at most $3.21^n/2$.

\item From Statement 1, we have $k_1^\ve \geq 1$. From Statement 2, we know there exists a matrix $\mathbf{M}$ in the form of Eq. \eqref{eq:mdef} with cond$(\mathbf{M}) \leq 3.21^{2d+1}/2$
Moreover, we have $\nu = d+1$ and  $U = 2d+1$. Substituting these values into the formula from Theorem \ref{thm:bitsize-grad} gives the result. \qedhere
\end{enumerate}
\end{proof} 

\revision{
\begin{corollary} (To Theorem \ref{thm:monomial-bit-eps-condensed}) \label{cor:univariate-bit-ndtau}
	Suppose that $\Sigma=\Sigma_{1,2d}$ and that, in the notation of Proposition \ref{thm:Nesterov}, the ordered bases $\vp$ and $\vq$ are the standard monomial bases of degree $d$ and $2d$, respectively. Let $\vt \in \Sc$, and assume that $\vt$ is an integer vector with $\tau = \log(\|\vt\|_{\infty})$. Then there exists an integer dual certificate $\overline{\vy} \in \Ssc \cap \zz^{2d+1}$ for $\vt$ whose components have bit size 
 	\begin{align*}
		\log\left( \|\overline{\vy}\|_{\infty}\right) &\approx \mathcal{O}\left(\tau d + d\log d\right).
	\end{align*}
\end{corollary}
\begin{proof}
%

The claim is a consequence of Theorem \ref{thm:monomial-bit-eps-condensed} with a suitable upper bound on $1/\varepsilon$ as a function of $\tau$ and $d$. (Here, as before, $\varepsilon > 0$ with $\vt - \varepsilon \ve \in \bd(\Sigma)$, with $\ve = \Lambda^*(\mathbf{I})$ is the coefficient vector of the polynomial $z\mapsto 1+z^2+\cdots+z^{2d}$.)

For the bound, notice that the largest $\varepsilon$ for which $t(x)-\varepsilon e(x) \geq 0$ for every $x$ has the property that the corresponding univariate polynomial $x \mapsto t(x)-\varepsilon e(x)$ has a multiple root, since its global minimum is zero. Therefore, the discriminant of this polynomial (with respect to $x$, treating $\varepsilon$ as a parameter) must vanish at the optimal $\varepsilon$. The discriminant is a univariate polynomial of $\varepsilon$, with integer coefficients whose bit sizes can be bounded from above using the (more general) bounds on the bit sizes of subresultant polynomials \cite[Proposition 8.46]{BasuPollackRoy2006}. Note that since $\ve=(1,0,1,\dots,0,1)$, if $\tau$ is an upper bound on the bit size of the coefficients of $t$ in the monomial basis, then the coefficients of the bivariate polynomial $t(x)-\varepsilon e(x)$ are of bit size at most $\tau+1$. Furthermore, treating the polynomial $t(x)-\varepsilon e(x)$ as a polynomial of $x$ whose coefficients are polynomials of $\varepsilon$, our polynomial is of degree $2d$, with coefficients of degree $1$. Thus, \cite[Proposition 8.46]{BasuPollackRoy2006} yields that the coefficients of the discriminant have bit sizes no larger than $\hat\tau := (4d-1)\big((\tau+1) + \log(2d) + \log(4d)\big) + \log(4d-1) = \Oh(\tau d + d\log(d))$. 

	We can now bound $1/\varepsilon$ from above using Cauchy's bound, 
	which yields that the bit size of $1/\varepsilon$ is at most $1+2\hat\tau = \Oh(\tau d + d\log(d))$. Substituting this bound and $\log(\|\vt\|_2) \leq \log(\|\vt\|_\infty \sqrt{2d+1}) = \Oh(\tau + \log(d))$ into the bound from Theorem \ref{thm:monomial-bit-eps-condensed} completes the proof.
\end{proof}
}

\revision{
These results may be compared to the bit sizes of the certificates obtained using the two algorithms analyzed in \cite{MagronSafeyElDinSchweighofer2019}. The first one finds certificates (explicit SOS decompositions) of bit size $\Oh(\tau \left(\frac{d}{2}\right)^{\frac{3d}{2}})$---linear in $\tau$, but exponential in the degree. The second one outputs a decomposition with coefficients of bit size $\Oh(\tau d^2 + d^3)$---also polynomial in the degree, a comparable result to ours above.
}

\subsection{Univariate polynomials over an interval}
We now consider polynomials nonnegative over an interval, which for simplicity we assume to be $[-1,1]$; all results in this section can be scaled appropriately to apply to any bounded interval. In this case, because the domain is bounded, the constant one polynomial $\vone$ belongs to $\Sc$, thus we may use $\ve=\vone$ instead of $\ve=\Lambda^*(\vI)$ as we do in the previous section. This leads to simpler and more interpretable results: $\varepsilon(\vt,\ve)$ in this case is simply the minimum value of $t$ over $[-1,1]$, which in turn can be bounded tightly using elementary techniques, without quantifier elimination. In the context of polynomial optimization, this reveals the rate at which the bit sizes of the certificates of lower bounds grow as the lower bounds approach the minimum value.

In this section, we also consider polynomial bases that are more commonly used in practical computation with high-degree polynomials than the monomial basis, namely Chebyshev polynomials (of the first kind) \cite[Sec.~3]{Trefethen2013} and interpolants \cite[Sec.~2]{Trefethen2013}.

The representations of even and odd degree polynomials over $[-1,1]$ vary slightly; we briefly recall the details for completeness. In the notation of Proposition \ref{thm:Nesterov}, for polynomials of degree $2d$, we use the weight polynomials $\vw(z)=\{1,1-z^2\}$ to represent $S_\vw=[-1,1]$, and regardless of the choice of bases $\vp_1, \vp_2$ and $\vq$, we have $m=2$, $U\defeq\dim(\Sigma^\vw_{1,2d})=2d+1$ and $\nu \defeq \sum_{i=1}^m \dim(\vp_i) = 2d+1$. For polynomials of degree $2d+1$, we use the weight polynomials $\vw(z)=\{1-z,1+z\}$, and we have $m=2$, $U\defeq\dim(\Sigma^\vw_{1,2d+1})=2d+2$ and $\nu = \sum_{i=1}^m \dim(\vp_i) = 2d+2$.


\subsubsection{Chebyshev polynomial basis}\label{sec:univ-cheb}

In this setting, we let $\vq$ be the basis of Chebyshev polynomials up to degree $2d$ or $2d+1$ (for even- or odd-degree polynomials, respectively). The $\vp_1$ and $\vp_2$ bases can be any bases of univariate polynomials of the appropriate degree. 

\begin{theorem}\label{thm:cheb-bit-eps-condensed} For univariate polynomials nonnegative over $[-1,1]$, represented in the Chebyshev basis, the following hold: \begin{enumerate} 
\item The constant $k_1^\vone$ is bounded below by 1.

\item There exists a matrix $\mathbf{M}$ of the form given in \eqref{eq:mdef} such that $\cond(\vM) \leq 4$. 

\item For every $\vt \in \Sigma^\circ$ of degree $2d$ or $2d+1$, with $\vt - \varepsilon\vone$ on the boundary of $\Sigma$, there exists an integer certificate $\overline{\vy}\in \Ssc$ for $\vt$ with 
\[
\|\overline{\vy}\|_{\infty} \leq\frac{1}{2} + \frac{2d+2}{\varepsilon}  \left\lceil 3\sqrt{2d+2}\|\vt\|_{2}\right\rceil.
\]
and the bit size of the largest component of $\overline{\vy}$ is bounded as 
\[
\log(\|\overline{\vy}\|_{\infty}) \approx \mathcal{O}\left( \log(d) + \log(\|\vt\|_{2}) + \log(1/\varepsilon)\right).
\]
\end{enumerate} 
\end{theorem} 

\begin{proof} For brevity, we include the details for even degree polynomials only. We will index all vectors, matrices, and point sets from $0$ to $2d$.    
\begin{enumerate} 
\item This result comes from Theorem 4.1 in \cite{DavisPapp2022}, wherein $k_1$ denotes our constant $k^\vone_1$.

\item Consider univariate polynomials of degree $2d$, and consider the points $\vz = \{z_0,\dots,z_{2d}\}$, with $z_l = \cos\left(\frac{\pi l}{2d}\right),$ 
so $\vz$ is the set of extrema of $q_{2d+1}(\cdot)$ (also known as the Chebyshev nodes of the second kind). Recall that $\mathbf{M}$ is defined by
\[
\mathbf{M} =  \sum_{l=0}^{s}\alpha_i\vq(z_l)\vq(z_l)^\T.
\]
for $s < \infty$ and for some real numbers $\alpha_0,\dots,\alpha_s$. Here, we will set $s = 2d$ and $\alpha_0 = \dots = \alpha_{2d} = 1$. 
From \cite{MasonHandscomb2003}, Chapter 4.6.1, equations 4.45-4.46c, we know that for all $i = 0,\dots, 2d$ and $j = 0,\dots, 2d$, we have 
\begin{equation*}
-\frac{1}{2}q_i(z_0)q_j(z_0) - \frac{1}{2}q_i(z_{2d})q_j(z_{2d}) + \sum_{l=0}^{2d}q_i(z_l)q_j(z_l) = \begin{cases} 0 \text{ when } i \neq j, \\ d \text{ when } i = j, i \neq 0 \text{ nor } 2d \\ 2d \text{ when } i = j = 0 \text{ or } 2d \end{cases}
\end{equation*}
Therefore, noting that $q_i(z_{0})q_j(z_{0}) = \left(-1\right)^{i+j}$ and  $q_i(z_{2d})q_j(z_{2d}) = 1$ for each pair $(i,j)$, we have
\begin{equation*}
\left(\mathbf{M}\right)_{ij} \defeq \left(\sum_{l=0}^{2d}\vq(z_l)\vq(z_l)^\T\right)_{ij} = \begin{cases} 
1 \text{ when } i \neq j \text{ and }i \equiv j \mod{2} \\
0 \text{ when } i \neq j \text{ and }i \not\equiv j \mod{2} \\
d + 1 \text{ when } i = j, i \neq 0 \text{ nor } 2d \\
2d + 1 \text{ when } i = j, i = 0 \text{ or } 2d.
\end{cases}
\end{equation*} 
To bound the condition number of $\mathbf{M}$ from above, it suffices to give a lower bound for $\lambda_{\min}\left(\mathbf{M}\right)$ and an upper bound for $\lambda_{\max}\left(\mathbf{M}\right)$ (since $\mathbf{M}$ is positive definite).
First, we will exhibit a lower bound for $\lambda_{\min}\left(\mathbf{M}\right)$. For any $\vx \in \rr^{2d+1}$, with $\|\vx\|_2 = 1$, we have 
\begin{align*}
\vx^\T \mathbf{M} \vx &= \left(\sum_{i=0}^{d-1} x_{2i+1}\right)^2 + \left(\sum_{i=0}^d x_{2i}\right)^2 + d\left(2x_0^2 + x_1^2 + x_2^2 + \dots + x_{2d-1}^2 + 2x_{2d }^2 \right) \\
&\geq d.
\end{align*}
Therefore, $\lambda_{\min}\left(\mathbf{M}\right) \geq d$. 

Recall that $\lambda_{\max}\left(\mathbf{M}\right)$ can be bounded by the  largest absolute row sum of $\mathbf{M}$. The largest absolute row sum of $\left(\mathbf{M}\right)$ is $(2d+1) + d = 3d + 1$. Hence $\lambda_{\max}\left(\mathbf{M}\right) \leq 3d + 1$. It follows that $\cond\left(\mathbf{M}\right) \leq \frac{3d+1}{d} \leq 4$.

\item From Statement 1, we have $k_1^\vone \geq 1$, and from Statement 2,  we know there exists a matrix $\mathbf{M}$ in the form of  \ref{eq:mdef}  with $\cond(\mathbf{M}) \leq 4$. Moreover, we have $\nu \leq 2d+2$ and $U \leq 2d + 2$. Substituting these values into the formula given in Theorem \ref{thm:bitsize-grad} gives the result. \qedhere
\end{enumerate} 
\end{proof}

Now we bound the minimum $\varepsilon$ of a positive univariate polynomial on the interval $[-1,1]$ with integer coefficients, so that we can give an $\varepsilon$-free result of the Theorem above. 

\begin{lemma}[\protect{adapted from \cite[Thm.~1.2]{BasuLeroyRoy2009}}]\label{thm:eps-univ-con}
Let $t$ be a univariate polynomial of degree $d$ taking only positive values on the interval $[-1,1]$, and suppose that the coefficients of $t$ in the monomial basis are integers of bit size no more than $\tau$. Then we have
\[ \min_{z\in[-1,1]} t(z) >
\frac{ 3^{d/2} }{ 2^{(2d-1)\tau} (d+1)^{2d-1/2} }.\]
\end{lemma}

Lemma \ref{thm:eps-univ-con} assumes that $t$ is represented in the monomal basis. For our next result, the change of basis (from the Chebyshev basis to monomial) can be incorporated using the observation that a polynomial of degree $d$ with integer coefficients of bit size at most $\tau$ in the Chebyshev basis also has integer coefficients in the monomial basis, and the bit size of the largest magnitude coefficient in the monomial basis is no more than $2d+\tau$. We are now ready to state our $\varepsilon$-free version of Theorem \ref{thm:cheb-bit-eps-condensed}.
\begin{corollary} (To Theorem \ref{thm:cheb-bit-eps-condensed}) Using the same notation as in Theorem \ref{thm:cheb-bit-eps-condensed}, 
assume that $\vt \in \Sc$ is the coefficient vector in the Chebyshev basis of a polynomial of degree at most $2d+1$, and assume that the components of $\vt$ are integers with bit sizes at most $\tau$. Then there exists an integer certificate $\overline{\vy}\in \Ssc \cap \zz^U$ for $\vt$ with  
\[
\log(\|\overline{\vy}\|_{\infty}) \approx  \mathcal{O}\left(d\tau + d^2\right).
\]
\end{corollary} 

\begin{proof} The result comes from substituting the bound on $\varepsilon$ from \mbox{Lemma \ref{thm:eps-univ-con}} and the previous paragraph into the bound from Theorem \ref{thm:cheb-bit-eps-condensed}.
\end{proof} 

\subsubsection{Univariate monomial basis}
Here, we let $\vq$ represent the univariate monomial basis up to degree $2d$ or $2d+1$, for even- or odd-degree polynomials, respectively. The bases $\vp_1$ and $\vp_2$ can be any univariate polynomial bases. 
\begin{theorem}\label{thm:univ-bit-eps-condensed} For univariate polynomials nonnegative over $[-1,1]$, represented in the monomial basis, the following hold: \begin{enumerate} 
\item The constant $k_1^\vone$ is bounded below by 1.

\item There exists a matrix $\mathbf{M}$ of the form given in \eqref{eq:mdef} such that $\cond(\mathbf{M}) \approx \mathcal{O}\left( (1+\sqrt{2})^{4U}/\sqrt{U}\right)$, wherein $U = \dim(\Sigma)$. 

\item For every $\vt \in \Sigma^\circ$ of degree $2d$ or $2d+1$, with $\vt - \varepsilon\vone$ on the boundary of $\Sigma$, there exists an integer certificate $\overline{\vy}\in \Ssc$ for $\vt$ with 
\[
\|\overline{\vy}\|_\infty \approx \mathcal{O}\left(d^{5/4}(1+\sqrt{2})^{4d+4}\frac{\|\vt\|_2}{\varepsilon}\right),
\]
and the bit size of the largest component of $\overline{\vy}$ is bounded as 
\[
\log(\|\overline{\vy}\|_{\infty}) \approx \mathcal{O}\left(d  + \log(\|\vt\|_{2}) + \log(1/\varepsilon)\right).
\]
\end{enumerate} 
\end{theorem} 

\begin{proof} \phantom{}
\begin{enumerate}
\item Using the observation that monomial basis polynomials also take values from $[-1,1]$ on the interval $[-1,1]$, this result comes from a slight adaptation of Theorem 4.1 in \cite{DavisPapp2022}, using the monomial basis instead of the Chebyshev. 

\item  We may choose $\vM$ to be the $U\times U$ Hilbert matrix, which is the (truncated) moment matrix of the uniform measure on $[0,1]$. Hence (similarly to the argument in Theorem \ref{thm:monomial-bit-eps-condensed}) it is also the moment matrix of a finitely supported measure on $[0,1]$, and therefore it can be written in the form given in \eqref{eq:mdef}. The Hilbert matrix is well-known to have a condition number which grows as $\mathcal{O}\left((1 + \sqrt{2})^{4U}/\sqrt{U}\right)$; see, for example, \cite[Thm. 294]{HardyLittlewoodPolya1934} and \cite[Eq. 3.35]{Wilf1970} for upper and lower bounds on the maximum and minimum eigenvalues of the Hilbert matrix, respectively. 


\item Substituting the results from Statements 1 and 2 as well as the bounds $\nu \leq 2d+2$ and $U \leq 2d + 2$ into the formula from Theorem \ref{thm:bitsize-grad} yields the result. \qedhere
\end{enumerate} 
\end{proof} 
As a corollary, we have the following: 
\begin{corollary} Using the same notation as in Theorem \ref{thm:univ-bit-eps-condensed}, assume that $\vt \in \Sigma^\circ$ is the coefficient vector in the monomial basis of a polynomial of degree at most $2d+2$, and assume that the components of $\vt$ are integers of bit size at most $\tau$. Then there exists an integer certificate $\overline{\vy} \in \Ssc \cap \zz^U$ for $\vt$ with 
\[
\log\left(\|\overline{\vy}\|_\infty\right) \approx \mathcal{O}\left(d\tau + d\log(d)\right).
\]
\end{corollary}

\begin{proof}
This comes from substituting the bound on $\varepsilon$ from Lemma \ref{thm:eps-univ-con} into the bound from Theorem \ref{thm:univ-bit-eps-condensed}.
\end{proof}

\subsubsection{Univariate Lagrange interpolant basis}
We now turn our attention to another common basis choice for polynomials over an interval: Lagrange interpolation polynomials. We use the same notation introduced at the beginning of Section \ref{sec:univ-cheb} to describe $\Sigma$, the weight polynomials $\vw$, and their respective variations for even- and odd-degree polynomials, except here we let $\vq$ (using the notation of Proposition \ref{thm:Nesterov}) be a \textit{Lagrange interpolation polynomial basis.} Precisely, let $\{z_1,\dots,z_U\}$ be a unisolvent point set in $S_\vw$. (As before, $U=\dim(\Sigma)$.) Define the Lagrange interpolation polynomial $q_i\; (i=1,\dots,U)$ to be the unique polynomial such that $q_i(z_i) = 1$ for each $i$ and $q_i(z_j) = 0$ when $i \neq j$. Then, we define the ordered Lagrange interpolation polynomial basis $\vq$ using these polynomials as $\vq = (q_1,\dots,q_U)$. 

The primary change from the Chebyshev and monomial bases is that the bit sizes of the certificates now depend on the choice of interpolation points $z_i$ in a quantifiable manner.

\begin{theorem}\label{thm:univ-interp-condensed} Suppose that, as detailed in the previous paragraph, we represent polynomials in the Lagrange interpolant basis corresponding to the interpolation points $\{z_1,\dots,z_U\}$. Then the following hold: 
\begin{enumerate} 
\item Letting $\mu = \max_{i=1,\dots,U}(\max_{-1\leq z \leq 1}|q_i(z)|)$, we have $k_1^\vone \geq \frac{1}{\mu}$. 
\item The matrix $\vM$ given in \eqref{eq:mdef} can be chosen to be the identity matrix, with $\cond(\vM) = 1$. 
\item For every $\vt \in \Sigma^\circ$ of degree $2d$ or $2d+1$ with $\vt - \varepsilon \ve \in \bd(\Sigma)$, there exists an integer dual certificate $\overline{\vy} \in \Ssc$ for $\vt$ with 
\end{enumerate} 
\deletethis{
\[
\|\overline{\vy}\|_{\infty} \leq \frac{1}{2} + \left\lceil \frac{3}{2}\sqrt{U}\|\vt\|_2 \right\rceil\left(\frac{\nu\mu}{\varepsilon} \right)
\]
}
\[
\|\overline{\vy}\|_{\infty} \leq \frac{1}{2} + \left\lceil \frac{3}{2}\sqrt{2d+2}\|\vt\|_2 \right\rceil\left(\frac{(2d+2)\mu}{\varepsilon} \right),
\]
whose largest component has bit size bounded as 
\begin{equation}\label{eq:interp-univ-bitsize}
\log(\|\overline{\vy}\|_{\infty}) \approx \mathcal{O}\left(\log(d) + \log(\|\vt\|_2) + \log(\mu) + \log(1/\varepsilon)\right).
\end{equation}
\end{theorem}

\begin{proof} \phantom{ }
\begin{enumerate} 
\item Recall that the optimization problem to compute $k_1^\vone$ can be solved by finding the minimum of the  $2U$ convex optimization problems in \eqref{eq:k1e-convex} (with two such problems for each $i = 1,\dots,U$). Fix $i$ with $1 \leq i \leq 
U$. The two optimization problems for this $i$  
have as their respective duals
\[\sup \{\pm x_i + |x_i| - \|\vx\|_{1} \ | \ \vone - \vx \in \Sigma, \vx \in \rr^{U}\}.\] In the univariate setting, a polynomial belongs to $\Sigma$ if and only if it is nonnegative. Hence, each $\vx = (0,\dots, \pm \frac{1}{\mu}, \dots, 0)$, with $\pm \frac{1}{\mu}$ in the $i$th coordinate, is a feasible solution to the dual problem with objective value $\frac{1}{\mu}$. Thus, $\frac{1}{\mu}$ is a lower bound for the infima of each of the $2U$ problems in \eqref{eq:k1e-convex}, therefore $k_1^\vone \geq \frac{1}{\mu}$.
\item By the definition of the Lagrange basis polynomials,
\[
 \vM = \sum_{i=1}^{U}\vq(\vz_i)\vq(\vz_i)^\T = \mathbf{I},
\]
the identity matrix, whose condition number is 1.
\item  From Statement 1, we have $k_1^\vone \geq \frac{1}{\mu}$, and from Statement 2, we know there exists a matrix $\mathbf{M}$ in the form of Eq. \ref{eq:mdef} with $\cond(\mathbf{M}) \leq 1$.  Moreover, we have $\nu \leq 2d+2$ and $U \leq 2d+2$.  Substituting these values into the formula given in Theorem \ref{thm:bitsize-grad} gives the result. \qedhere
\end{enumerate} 
\end{proof} 

\begin{remark} The parameter $\mu$ from Statement 1 of Theorem \ref{thm:univ-interp-condensed} is closely related to the \emph{Lebesgue constant} $\mathcal{L} \defeq \max_{x \in [-1,1]} \sum_{i=1}^{U}|q_i(x)|$, the operator norm of the interpolation operator (with respect to the uniform norm). It is well-understood that the choice of interpolation points with a small Lebesgue constant is crucial in numerical computation with interpolants (see, e.g., \cite[Chap.~15]{Trefethen2013}); this has also been demonstrated in the context of sum-of-squares optimization \cite{Papp2017}. One interpretation of Theorem \ref{thm:univ-interp-condensed} is that the choice of interpolation points is important even in exact-arithmetic computation, as the bit sizes of the certificates is affected by the Lebesgue constant.

In the univariate case, the Lebesgue constant with $U$ suitably chosen interpolation points from $[-1,1]$ can be as low as $\Oh(\log U)$ \cite[Chap.~15]{Trefethen2013}. Thus, even for fairly suboptimal points, the impact of $\mu \leq U\mathcal{L} \approx \Oh(d \log d)$ in the bit size bound \eqref{eq:interp-univ-bitsize} is dominated by the $\log d$ term, simplifying the bound to \[\log(\|\overline{\vy}\|_{\infty}) \approx \Oh(\log(d) + \log(\|\vt\|_2) + \log(1/\varepsilon)).\]


\end{remark}

\subsection{Multivariate polynomials over a bounded set} \label{sec:multi-interp}
In this section, we assume that $\Sigma = \Sigma_{n,\vd}^\vw$, with $\vw$ a set of weight polynomials describing a bounded set $S_\vw\subseteq \R^n$. We assume that, using the notation of Proposition \ref{thm:Nesterov}, $\vq$ is represented by a multivariate Lagrange interpolation basis.
Although in the multivariate case we can no longer rely on the fact that nonnegative polynomials are the same as WSOS polynomials with suitably chosen weights, the analysis in the multivariate case can be made largely identical to the univariate analysis of the previous section under an additional assumption that is only slightly stronger than assuming $S_\vw$ to be bounded \revision{(see also Remark \ref{rem:mild-assumptions} below)}.



\begin{theorem}
Suppose the $\vq$ basis polynomials (in the notation of Proposition \ref{thm:Nesterov}) are the Lagrange basis polynomials corresponding to the (unisolvent) interpolation points $\{\vz_1,\dots,\vz_U\}\subseteq S_\vw$. In addition, suppose that there is a $\mu>0$ such that $\mu + q_i\in\Sigma$ and $\mu - q_i\in\Sigma$ (for each $i=1,\dots,U$). Then the following hold:
\begin{enumerate} 
\item For $k_1^\vone$ defined as in \eqref{eq:k1def}, we have $k_1^\vone \geq \frac{1}{\mu}$.
\item The matrix $\vM$ given in \eqref{eq:mdef} can be chosen to be the identity matrix, with $\cond(\vM) = 1$.
\item For all $\vt \in \Sigma^\circ$ and $\varepsilon > 0$ with $\vt - \varepsilon\vone \in \bd(\Sigma)$, there exists an integer certificate $\overline{\vy}\in \Ssc \cap \zz^U$ for $\vt$ with
\begin{align*}
\|\overline{\vy}\|_\infty &\leq \frac{1}{2} +  \left\lceil\frac{3}{2}\sqrt{U}\|\vt\|_{2}\right\rceil\left(\frac{\mu\nu}{\varepsilon}\right) 
\end{align*}
whose largest component has bit size bounded as 
\[ 
\log(\|\overline{\vy}\|_{\infty}) \approx \mathcal{O} \left( \log(U) + \log(\|\vt\|_2) + \log(\mu) + \log\left(\nu \right) + \log(1/\varepsilon) \right).
\]
\end{enumerate}
\end{theorem}
\begin{proof}
The proofs of the first two statements are essentially identical to those in Theorem \ref{thm:univ-interp-condensed}. Substituting those two bounds into the formula given in Theorem \ref{thm:bitsize-grad} gives the third claim.
\end{proof}

\begin{remark} \label{rem:mild-assumptions}
The new assumption $\mu\pm q_i\in\Sigma$ is relatively mild, given that $S_\vw$ is bounded. Recall that if $S_\vw$ is bounded, then every strictly positive polynomial over $S_\vw$ belongs to the interior of the cone of polynomials nonnegative over $S_\vw$; in particular, the constant $1$ polynomial belongs to the interior. It is reasonable to assume that $1 \in \Sc$ holds as well (as it automatically does for many WSOS cones commonly encountered in applications), in which case $1 \pm \frac{1}{\mu}q_i \in \Sigma$ automatically holds for every large enough $\mu$.

The condition $1 \in \Sc$ plays an important role in the context of dual certificates. For example, it ensures that \emph{every} polynomial in $\operatorname{span}(\Sigma)$ has a WSOS lower bound \cite[Lemma 3.1]{DavisPapp2022}; as such, we will rely on it when we discuss exact-arithmetic algorithms to compute rational dual certificates in Section \ref{sec:computing}. As discussed in \cite[Theorem 3.7]{DavisPapp2022}, even in the case when this assumption does not hold, it is possible to extend $\Sigma$, with the inclusion of a single additional weight that is nonnegative on $S_\vw$, to satisfy this condition without changing $\operatorname{span}(\Sigma)$ (in particular, without increasing the degrees or invoking a Postivstellensatz).
\end{remark}

\section{Computing certified WSOS lower bounds in rational arithmetic}\label{sec:computing}
Having established the existence of rational dual certificates with a priori bounded bit sizes, we now turn to the question of computing such certificates. More precisely, given a polynomial $\vt$ and a tolerance $\varepsilon>0$, we want to compute a rational lower bound $c$ that lies between the optimal WSOS lower bound $c^*$ and $c^*-\varepsilon$, along with a rational dual certificate (of a small bit size) proving $\vt-c\vone\in\Sigma$.

The new algorithm (Algorithm \ref{alg:Newton2} below) is an adaptation of Algorithm~1 from \cite{DavisPapp2022}, which is a hybrid method for the solution of the same problem. Since that algorithm was designed to run in finite-precision floating point arithmetic, the quality of the lower bound is limited by the precision of the arithmetic (limiting how small $\varepsilon$ can be). That algorithm, as is, cannot be efficiently implemented in infinite precision (rational) arithmetic, because the bit sizes of the intermediate quantities (and the returned certificate) blow up as the algorithm progresses even if $\varepsilon$ is large. The new Algorithm \ref{alg:Newton2} follows the same blueprint, but rounds all intermediate quantities (lower bounds and certificates) to ``nearby'' rational ones with small denominators, while maintaining the desirable properties of the original algorithm that guarantee that the new algorithm also converges linearly to the optimal WSOS lower bound.

The algorithm works for almost any WSOS cone $\Sigma$; our only assumption is that $\vone\in\Sc$. This is a mild assumption that ensures that every polynomial has a WSOS lower bound; recall Remark \ref{rem:mild-assumptions}.

\subsection{The algorithm}\label{sec:algorithm}
\afterpage{
\begin{algorithm2e}
  \DontPrintSemicolon
  \SetKwInOut{Input}{input}
  \SetKwInOut{Params}{parameters}  
  \SetKwInOut{Output}{outputs}
  \Input{A polynomial $\vt$; a tolerance $\varepsilon>0$.}
  \Output{A lower bound $c$ on the optimal WSOS lower bound $c^*$ satisfying $c^*-c\leq \varepsilon$; a dual vector $\vx\in\Ssc$ certifying $\vt-c\vone\in\Sigma$.} 
  \Params{An oracle for computing the barrier Hessian $H$ for $\Sigma$; a radius $r\in(0,1/4]$, a radius $r_N$ satisfying $ \frac{r^2}{1-2r} < r_N < \frac{r}{1+2r}$, a certificate $\vx$ satisfying $\|-g(\vx) - \vone\|_{\vx}^* \leq \frac{r}{r+1}$.
  }

  \medskip

  Compute $c_0 := -\left(\frac{r}{r+1} - \|-g(\vx) - \vone\|_{\vx}^*\right)^{-1}\|\vt\|^*_{\vx}.$ Set $c := c_0$ and $\mathbf{x} := -\frac{1}{c_0}\vx$.
  
  \Repeat{$\Delta c \leq \frac{1}{2}\rho C\varepsilon$ \do}
  {
    Set $\vx_+ := 2\vx - H(\vx)^{-1}(\vt -c\vone )$. \label{line:x-update}
    
    Round $\vx_+$ component-wise to a point $\vx_N$ with denominators $N \defeq \Big\lceil\frac{\sqrt{U}}{2}\big(\frac{1 + r_N}{r_N - \frac{r^2}{1-2r}}\big)\|H(\vx_+)^{1/2}\|\Big\rceil$.
    \deletethis{{\color{red}
    $N \defeq \left\lceil\frac{\sqrt{U}}{2}\left(\frac{(1-2 r) (r_N+1)}{-r^2-2 r\cdot r_N+r_N}\right)\|H(\vx_+)^{1/2}\|\right\rceil$. Checked this. Alt $r,r_N$ constant (avoids cdot): $\frac{(2r-1)(r_N+1)}{r^2 + r_N(2r-1)}$, or alt which is straightforward: $\frac{1 + r_N}{r_N - \frac{r^2}{1-2r}}$}
    }\label{line:round-x}
    
    
    Solve for $c_+$ the scalar quadratic equation
    \[
    \|\vx_N - H(\vx_N)^{-1}(\vt - c_+\vone)\|_{\vx_N} = \frac{r}{r+1},
    \]\label{line:c-update}
and set $c_+$ equal to the larger of the two solutions. 
    
    Set $\overline{\Delta c} := c_+ - c$. Chose $c_+'$ to be the rational point in the interval $[c + \frac{\overline{\Delta c}}{2}, c_+]$ with the smallest possible denominator.  \label{line:c-round}
    
    Set $\Delta c := c_+' - c$.
    Set $c := c_+'$.
    Set $\vx := \vx_N$. \label{line:stopping}
    }
    \Return{$c$ and $\vx$.}
    \caption{Compute the best WSOS lower bound and a dual certificate}\label{alg:Newton2}
\end{algorithm2e}
} 

\deletethis{NOTE: In Algorithm \ref{alg:Newton2}, we need $\frac{r^2}{1-2r} \leq r_N$ because of Lemma \ref{thm:rounding}, and we need $r_N   \leq \frac{r}{1+2r}$ because of Lemma  \ref{thm:constant-update} -- see deleted lemmas in analysis. 
This works out when $0 < r < \frac{1}{r} \left( \sqrt{17} 3 \right) \approx .28$. Since we already have $r \leq \frac{1}{4}$, this is okay. 

Note that $r_N$ being chosen to be closer to $\frac{r^2}{1-2r}$  makes the denominators in Line \ref{line:round-x} huge (since the denominator is close to 0). On the other hand, choosing $r_N$ to be closer to $\frac{r}{1+2r}$ mean that we don't get much improvement in the $c$ to $c_+$ step (see Lemma \ref{thm:constant-update} {\color{red} although this Lemma is deleted}). 

Note that $\vx_N$ rounding step in Line \ref{line:round-x} uses Lemma \ref{thm:rounding}  with {\color{blue} $r_2 = r_N$ and $r_1 = \frac{r^2}{1-2r}$.}}

The pseudocode of the algorithm is shown in Algorithm \ref{alg:Newton2} below\revision{; see a detailed example of one iteration of the algorithm after the outline of its analysis, in Example \ref{ex:alg1}}. Throughout, $\vx$ represents a dual certificate for the polynomial $\vt-c\vone$, where $c$ is the current certified WSOS lower bound and $\vt$ is the input polynomial we wish to bound. In its main loop, the algorithm first updates the current certificate $\vx$ to be closer to the gradient certificate for the current $\vt-c\vone$ (Line~\ref{line:x-update})\revision{, by taking a single Newton step towards the solution of the nonlinear system $-g(\vx) = \vt-c\vone$}. This updated certificate is then rounded to a rational one with smaller denominators in Line~\ref{line:round-x}. \revision{(The matrix norm required for this calculation could be expensive to compute, but since any upper bound can be substituted for $\|H(\vx_+)\|$, one can get an acceptable rigorous bound by using the Frobenius norm instead, which is easy to compute exactly even in rational arithmetic.)} Then the lower bound $c$ is improved in Line~\ref{line:c-update} and rounded to a nearby rational bound that is still certified by the same rounded dual certificate (Line~\ref{line:c-round}). This last step can be implemented efficiently using continued fractions or Farey sequences. Alternatively, it could be replaced by a naive rounding to the largest number in the interval with denominator $\lceil 2/\overline{\Delta c}\rceil$, the analysis below remains valid even in that case.

In addition to the notation introduced in the algorithm, we will use the following notation throughout the rest of the section. We let $\vy$ be the vector satisfying $-g(\vy) = \vt - c\vone$ and $\vy_+$ be the vector satisfying $-g(\vy_+) = \vt - c_+ \vone$. The constants $\rho$ and $C$ used in the termination criterion will be precisely defined and justified below, in Theorem \ref{thm:algorithm}---for now, we may treat the main loop of the algorithm as an infinite loop.

The initialization of the algorithm requires a certificate $\vx$ satisfying \mbox{$\|-g(\vx) - \vone\|_{\vx}^* \leq \frac{r}{r+1}$.} This may be readily available, for example, when the gradient certificate of $\vone$ is known in closed form. (See \cite[Example 4]{DavisPapp2022} for an example.) If such a certificate $\vx$ is not known, we may run Algorithm~\ref{alg:initialization2} discussed in Section~\ref{sec:initialization} to  compute such a certificate. Note that this initial vector is independent of the coefficient vector $\vt$, and only needs to be computed once for every WSOS cone $\Sigma$.

The analysis of Algorithm \ref{alg:Newton2} here is similar to the analysis of Algorithm~1 in \cite[Section 3.3]{DavisPapp2022}. Therefore, we only give a concise outline of the proofs of its correctness and rate of convergence. We shall refer to the analysis of its predecessor whenever possible, focusing on where the analyses differ as a result of the rounding steps.

\begin{theorem}\label{thm:algorithm} Suppose that, at the beginning of the main loop of Algorithm~\ref{alg:Newton2}, $\|\vx-\vy\|_\vx \leq r$ for some $r<\frac{1}{4}$. Then:
\begin{enumerate} 
\item 
After Step \ref{line:x-update}, $\|\vx_+ - \vy\|_{\vx_+}  \leq \frac{r^2}{1 - 2r}$.
\item 
After Step \ref{line:round-x}, $\|\vx_N - \vy\|_{\vx_N} \leq r_N$. 
\item 
After Steps \ref{line:c-update} and \ref{line:c-round}, we have $c_+'>c$ and $\|\vx_{N} - \vy_+\|_{\vx_N} \leq r$, so $\|\vx-\vy\|_\vx \leq r$ also holds at the end of the loop, and the algorithm improves the lower bound $c$ in each iteration.
\end{enumerate}
Moreover, Algorithm \ref{alg:Newton2} is globally $q$-linearly convergent to $c^* = \max\{c \ | \ \vt - c\vone \in \Sigma\}$, the optimal WSOS lower bound for the polynomial $\vt$. More precisely, in each iteration of Algorithm \ref{alg:Newton2}, the improvement of the lower bound $\Delta c = c_+' - c$ satisfies
\begin{equation}\label{eq:linconvrate}
\frac{\Delta c}{c^* - c} \geq \frac{1}{2}\rho C, 
\end{equation}
with the absolute constant $\rho \defeq \frac{r}{r+1} - \frac{r_N}{1-r_N}$ and the $\Lambda$-dependent constant $C > 0$ defined as in Theorem 3.5 of \cite{DavisPapp2022}.
\end{theorem}

\begin{proof} Statement 1 is identical to \cite[Lemma 3.2]{DavisPapp2022}, and statement 2 comes from Lemma \ref{thm:rounding} with $r_1 = \frac{r^2}{1-2r}$ and $r_2 = r_N$.

Statement 3 is analogous to \cite[Lemma 3.3]{DavisPapp2022}, replacing $\vx_+$ therein with $\vx_N$ and making use of the fact that $\|\vx_N - \vy\|_{\vx_N} \leq r_N$ whenever \cite[Lemma 3.3]{DavisPapp2022} uses the inequality $\|\vx_+ - \vy\|_{\vx_+} \leq \frac{r^2}{1-2r}$. The fact that $c_+' > c$ comes from the fact that $c_+ > c$ \cite[Lemma 3.3]{DavisPapp2022} and the construction of $c_+'$.

The linear convergence result is analogous to \cite[Theorem 3.6]{DavisPapp2022}, with $c'_+$ playing the role of $c_+$. The rounding down of the lower bound $c_+$ to $c'_+$ in Step \ref{line:c-update} ensures that in spite of the rounding, the progress our Algorithm \ref{alg:Newton2} is at least half of what the progress would be without rounding (as in \cite[Algorithm 1]{DavisPapp2022}) The final inequality \eqref{eq:linconvrate} justifies the termination criterion in Line \ref{line:stopping}: if $\Delta c\leq \frac{1}{2}\rho C\varepsilon$, then the gap between the certified and the optimal lower bound is $c^*-c \leq \varepsilon$ as wanted.
\end{proof} 

\revision{
\begin{example}\label{ex:alg1}
Consider the polynomial $t(z_1,z_2) = 3z_1^2-6z_1 z_2+z_2^2+2z_1-z_2$. Its global minimum on the unit disk is $c^*\approx -1.70768680307$. We can certify a sequence of lower bounds $c<c^*$ by writing $t-c$ in the form
\[t(z_1,z_2) - c = \sigma_1(z_1,z_2) + (1-z_1^2-z_2^2)\sigma_2,\]
where $\sigma_1$ is a quadratic SOS polynomial and $\sigma_2$ is a nonnegative constant (a degree-$0$ SOS polynomial); i.e. by showing that $t-c\in\Sigma^{(1,w)}_{2,(2,0)}$, wherein $w$ is the weight polynomial  $(z_1,z_2)\mapsto (1-z_1^2-z_2^2)$.

Using the monomial basis to represent all polynomials, we have the coefficient vector $\vt =  (0,2,3,-1,-6,1)$, and it is straightforward to construct the $\Lambda$ operator: we have $U=\dim(\Sigma)=6$, $(L_1,L_2)=(3,1)$, and $\Lambda = \Lambda_1 \oplus \Lambda_2$, where
\begin{equation}\label{eq:ex-Lambda}
\Lambda_1(\vx) = \left(
\begin{array}{ccc}
 x_1 & x_2 & x_4 \\
 x_2 & x_3 & x_5 \\
 x_4 & x_5 & x_6 \\
\end{array}
\right),
\quad
\Lambda_2(\vx) = \left(
\begin{array}{c}
 x_1-x_3-x_6 \\
\end{array}
\right)
\end{equation}
Now the Hessian of the barrier function can be computed efficiently using the formula \eqref{eq:H}.

To initialize Algorithm \ref{alg:Newton2}, we need a vector $\vx_1$ sufficiently close to the gradient certificate of the constant one polynomial. In this simple example, we will use the gradient certificate itself, which can be computed in closed form and happens to be a rational vector, $\vx_1 = (4,0,\frac{4}{3},0,0,\frac{4}{3})$. This can be verified by direct computation: $-g(\vx_1) = (1,0,0,0,0,0) = \vone$.

In the main loop of the algorithm, we shall use the parameters $r=1/4$ and $r_N = 1/7$ in this example. The corresponding (very crude) initial lower bound $c_0 = -\frac{10}{3}\sqrt{\frac{410}{3}}\approx -38.97$ is not rational, but we can round it down to $c_0=-39$. We can also verify that this lower bound is indeed certified by $\vx_1$: using the definition of dual certificates, it suffices to compute $H(\vx_1)^{-1}(\vt-c_0\vone) =  \left( \frac{484}{3},\frac{16}{3},\frac{1532}{27},-\frac{8}{3},-\frac{16}{3},\frac{1436}{27}\right)$ and confirm that both $\Lambda_i$ from \eqref{eq:ex-Lambda} are indeed positive definite. For completeness, we can also generate an explicit WSOS representation using \eqref{eq:S-def}:
\[ t(z_1,z_2) + 39 = \begin{pmatrix}1 \\ z_1 \\ z_2\end{pmatrix}^\T
\begin{pmatrix}
 \frac{121}{12} & 1 & -\frac{1}{2} \\
 1 & \frac{383}{12} & -3 \\
 -\frac{1}{2} & -3 & \frac{359}{12} \\
\end{pmatrix}
\begin{pmatrix}1 \\ z_1 \\ z_2\end{pmatrix} + (1-z_1^2-z_2^2)\frac{347}{12}. \]

\noindent The first iteration of the algorithm proceeds as follows:
\begin{enumerate}
\item After computing $H(\vx_1)$, the updated certificate in Line \ref{line:x-update} is
\[\vx_+ = \left(\frac{452}{4563},\frac{8}{4563},\frac{1372}{41067},-\frac{16}{4563},\frac{16}{4563},\frac{1276}{41067}\right).\]
\item The denominator for the ``compressed'' certificate is $N=5029$; the computation of this involves calculating the updated Hessian, $H(\vx_+)$, the rest is trivial arithmetic. Rounding $\vx_+$ componentwise, the rounded certificate becomes
\[\vx_N = \frac{1}{5029}(498, 9, 168, -18, 18, 156).\]
\item To update the lower bound, we construct the scalar quadratic equation in Line \ref{line:c-update}; reusing the already computed Hessian, this is simple arithmetic. The equation can be written as
\[  29\,387\,195\,615\,576+1\,508\,777\,838\,050 c_+ + 19\,170\,557\,325 c_+^2  = 0, \]
whose larger root is approximately $-35.4$, meaning that (in the spirit of keeping the denominators as small as possible), we can update our lower bound to $c'_+ = -36$. Indeed, we have
\[ t(z_1,z_2) + 36 = \begin{pmatrix}1 \\ z_1 \\ z_2\end{pmatrix}^\T
\begin{pmatrix}
  s_1 & 1 & -\frac{1}{2} \\
 1 & s_2 & -3 \\
 -\frac{1}{2} & -3 & s_3 \\
\end{pmatrix}
\begin{pmatrix}1 \\ z_1 \\ z_2\end{pmatrix} + (1-z_1^2-z_2^2) s_4 \]
with $(s_1,s_2,s_3,s_4) = \frac{1}{344\,769}(3\,203\,164, 10\,242\,827, 9\,553\,289, 9\,208\,520)$.
\end{enumerate}

Although some of the coefficients appear frighteningly large for a toy example, it shall be emphasized that the explicit WSOS decompositions of $t-c$ need not be computed in the algorithm.

Continuing with the algorithm, the bit sizes of the dual certificates appear to grow linearly with the number of iterations, and (as predicted by the theory), the difference $c^*-c$ decreases exponentially with the number of iterations. For instance, after 200 iterations, the lower bound is  $-1\,579\,834/925\,131$, only about $2\cdot 10^{-10}$ away from the true minimum value.
\end{example}
}

\subsection{Initialization}\label{sec:initialization}
Algorithm \ref{alg:Newton2} requires a suitable certificate of $\vone$, the constant one polynomial, to initialize. Such a certificate may be available either in closed form (e.g., for cones of univariate polynomials nonnegative on an interval, the gradient certificate can be determined analytically \cite[Example 4]{DavisPapp2022}), or from ``preprocessing'', e.g., when other polynomials from the same space $\operatorname{span}(\Sigma)$ have already been bounded.

If we do not know a suitable $\vx$ to start with, then we could attempt to find the gradient certificate (or a rational approximation of it) by solving the system  $-g(\vx) = \vone$ by a general-purpose method for polynomial systems. Alternatively (and more efficiently), we can find an approximate solution to this system by numerically solving the convex optimization problem
\[ \min \{f(\vx) + \vone^\T \vx\,|\,\vx\in\Sigma^*\}. \]

Instead of these numerical approaches, we can also leverage Algorithm~\ref{alg:Newton2} itself to find a suitable initial point. Suppose we have an interior point $\vx \in \Ssc$, which is by definition the gradient certificate of the polynomial $\vs = -g(\vx)$. Then we can apply Algorithm~\ref{alg:Newton2} starting with this initial pair ``in reverse,'' computing a sequence of certificates for polynomials of the form $\vs+c\vone$ for \emph{increasing} values of $c$. The same certificates in turn certify $c^{-1}\vs+\vone$ as well, which is approximately the same as the polynomial $\vone$ when $c$ is large enough. The details of this approach are presented in Algorithm~\ref{alg:initialization2}. Its analysis largely follows the steps laid out in Theorem \ref{thm:algorithm}, with two minor adjustments regarding the progress and the termination criterion. The change in Line~\ref{line:init-c-round} guarantees that the stopping criterion can be met, as we show later in Lemma~\ref{thm:init-termination}.

\afterpage{
\begin{algorithm2e}
  \DontPrintSemicolon
  \SetKwInOut{Input}{input}
  \SetKwInOut{Output}{outputs}
  \SetKwInOut{Params}{parameters}
  \Input{A vector $\vx \in \Ssc$.}
  \Params{An oracle for computing the barrier Hessian $H$ for $\Sigma$; a radius $r\in(0,1/4]$, a radius $\frac{r^2}{1-2r} < r_N < \frac{r}{1+2r}$.}
  \Output{A certificate $\vx \in \Ssc$ satisfying $\|-g(\vx) - \vone\|_{\vx}^* \leq \frac{r}{r+1}$.}

  \medskip

  Compute $\vs := -g(\vx)$. Set $c:=0$. 
  
  \Repeat{$\|-g(c\vx) - \vone\|^*_{c\vx} \leq \frac{r}{r+1}$. \do}
  {
    Set $\vx_+ := 2\vx - H(\vx)^{-1}(\vs + c\vone )$. \label{line:init-x-update}
    
    Round $\vx_+$ component-wise to a point $\vx_N$ with denominators
    $N \defeq \Big\lceil\frac{\sqrt{U}}{2}\big(\frac{1 + r_N}{r_N - \frac{r^2}{1-2r}}\big)\|H(\vx_+)^{1/2}\|\Big\rceil$. \label{line:init-round-x}
    
    
    Solve for $c_+$ the scalar quadratic equation
    \[
    \|\vx_N - H(\vx_N)^{-1}(\vs + c_+\vone)\|_{\vx_N} = \frac{r}{r+1},
    \]\label{line:init-c-update}
    and set $c_+$ equal to the larger of the two solutions. 
    
    Set $\overline{\Delta c} := c_+ - c$. Chose $c_+'$ to be the rational point in the interval $[c + \frac{1}{2}\overline{\Delta c}, c + \frac{2}{3}\overline{\Delta c}]$ with the smallest possible denominator.  \label{line:init-c-round}
    
    Set $\Delta c := c_+' - c$.
    Set $c := c_+'$.
    Set $\vx := \vx_N$.
    \label{line:init-stopping}
}
    \Return{$c\vx$.}
    \caption{Initialization for Algorithm \ref{alg:Newton2}}\label{alg:initialization2}
\end{algorithm2e}
} 

Analogously to Section \ref{sec:algorithm}, we will let $\vy$ be the vector satisfying $-g(\vy) = \vs + c\vone$ and $\vy_+$ be the vector satisfying $-g(\vy_+) = \vs + c_+ \vone$ throughout this section, in addition to the notation introduced in Algorithm~\ref{alg:initialization2}. The algorithm can be initialized by any point in the interior of $\Sigma^*$, and by definition, at the beginning of the main loop of the algorithm, we have $\vx=\vy$, and therefore $\|\vx-\vy\|_\vx \leq r$. The proof of Theorem \ref{thm:algorithm} can be repeated almost verbatim to show the following.
\begin{theorem} Suppose that, at the beginning of the main loop of Algorithm~\ref{alg:initialization2}, $\|\vx-\vy\|_\vx \leq r$ for some $r<\frac{1}{4}$. Then:
\begin{enumerate} 
\item After Step \ref{line:init-x-update}, $\|\vx_+ - \vy\|_{\vx_+}  \leq \frac{r^2}{1 - 2r}$.
\item
After Step \ref{line:init-round-x}, $\|\vx_N - \vy\|_{\vx_N} \leq r_N$. 
\item After Steps \ref{line:c-update} and \ref{line:c-round}, we have $c_+'>c$ and $\|\vx_{N} - \vy_+\|_{\vx_N} \leq r$, so $\|\vx-\vy\|_\vx \leq r$ also holds at the end of the loop. The increase $\Delta c = c_+' - c$ satisfies
\begin{equation*}
\frac{\Delta c}{c - c^*} \geq \frac{1}{2}\rho C, 
\end{equation*}
with the same constants $\rho$ and $C$ as in Theorem~\ref{thm:algorithm}. Thus, the constant $c$ increases exponentially as the algorithm progresses.
\end{enumerate}
\end{theorem}

By definition, if Algorithm \ref{alg:initialization2} terminates, it returns a vector that can be used as an initial vector $\vx$ in Algorithm \ref{alg:Newton2}. It only remains to show that the algorithm indeed terminates. We shall show this in two steps. First, we show that as the algorithm progresses and $c$ increases, $\|\vs\|^*_{\vx}$ tends to zero. Then we argue that this ensures that the polynomial $c^{-1}\vs+\vone$ is eventually ``close enough'' to $\vone$ that $c\vx$ is sufficiently close to the gradient certificate of $\vone$. 

\begin{lemma}\label{thm:init-termination} Let $\vx \in \Ssc$ be the certificate of the polynomial $\vs + c\vone \in \Sc$ as defined in Algorithm \ref{alg:initialization2}. Then 
\begin{enumerate} 
\item $\|\vs\|^*_{\vx}$ tends to $0$, and 
\item Algorithm \ref{alg:initialization2} terminates. 
\end{enumerate}
\end{lemma}

\begin{proof} We begin with the first statement. Let $\vx_1$ be the gradient certificate of $\vone$ and fix an arbitrary $a\in(0,1)$. Let $\vy$ be the gradient certificate of $\vs + c\vone$. By \cite[Lemma 3.1]{DavisPapp2022}, we know that for every sufficiently large $c$,
\[\|c^{-1}\vx_1 - \vy\|_{c^{-1}\vx_1} \leq a.\] 
Then, using inequality \eqref{eq:lemma4} from the Appendix and the fact that $\|\vs\|_{c^{-1}\vx_1}^* = c^{-1}\|\vs\|_{\vx_1}^*$, we have 
\begin{equation}\label{eq:alg-init-c-exists}
    \|\vs\|_{\vy}^* \overset{\eqref{eq:lemma4}}{\leq} \frac{\|\vs\|_{c^{-1}\vx_1}^*}{1 - \|c^{-1}\vx_1 - \vy\|_{c^{-1}\vx_1}  } \leq \frac{c^{-1}\|\vs\|_{\vx_1}^*}{1 - a}.
\end{equation}
Using inequality \eqref{eq:lemma4} again, we have
\begin{equation}\label{eq:alg-init-c-exists-2}
    \|\vs\|_{\vx}^* \overset{\eqref{eq:self-concordance}}{ \leq} \frac{\|\vs\|_{\vy}^*}{1 - \|\vx - \vy\|_{\vy}}  \leq \frac{\|\vs\|_{\vy}^*}{1 - r}.
\end{equation}
Therefore, we have
\[
\|\vs\|_{\vx}^* \overset{\eqref{eq:alg-init-c-exists},\eqref{eq:alg-init-c-exists-2}}{\leq} \frac{c^{-1}\|\vs\|_{\vx_1}^*}{(1 - r)(1-a)}.
\]
Since $a$ can be chosen to be arbitrarily close to $0$, and since $r$ and $\|\vs\|_{\vx_1}^*$ are constants, it follows that as $c \to \infty$, $\|\vs\|_{\vx}^* \to 0$. 

Now, we show that Algorithm \ref{alg:initialization2} terminates. At the end of each iteration, we have
\begin{equation*}
\begin{aligned}
\left\|-g(c\vx) - (c^{-1}\vs + \vone)\right\|_{c\vx}^* &\overset{\eqref{eq:log-homogeneity}}{=}
\|-g(\vx) - (\vs + c\vone)\|_{\vx}^* \\ &\overset{\eqref{eq:gH-identities}
}{=}
\|\vx - H(\vx)^{-1}(\vs + c\vone)\|_\vx \overset{\text{Line } \ref{line:init-c-round}}{<} \frac{r}{r+1}. 
\end{aligned}
\end{equation*}
Thus, we have
\begin{equation*}
\begin{aligned}
    \|-g(c\vx) - \vone\|_{c\vx}^* &\leq \left\|-g(c\vx) - (c^{-1}\vs + \vone)\right\|_{c\vx}^* + \left\|(c^{-1}\vs + \vone) - \vone\right\|_{c\vx}^* \\
    &< \frac{r}{r+1} +\left\|\vs\right\|_{\vx}^*.
\end{aligned}
\end{equation*}
By Statement 1, $\|\vs\|^*_{\vx}$ tends to 0, so eventually the stopping criterion is satisfied. \end{proof}

\section{Discussion}\label{sec:discussion}
\paragraph{Bit size bounds on certificates from Algorithm \ref{alg:Newton2}}
We opted to separate the discussion on the bit sizes of the certificates and Algorithm \ref{alg:Newton2}. In principle, one could study the former question ``constructively'' by analyzing the bit sizes of the certificates computed by the algorithm, but we think it is useful to underline that both the concept of dual certificates and the bit size bounds are independent of any particular algorithm. Theorem~\ref{thm:grad-denominators} and Lemma~\ref{thm:roundedynorm} are both derived assuming that the dual certificate $\vy$ at hand is the gradient certificate for simplicity of presentation; both of these results can be easily adapted to the setting where $\vy$ is any dual certificate that is sufficiently close to the gradient certificate. Similarly, any algorithm that computes a dual certificate by computing a vector $\vx$ sufficiently close (in the local $\vx$-norm) to the gradient certificate $\vy$ will produce certificates with boundable bit sizes.

\revision{
\paragraph{The Christoffel-Darboux polynomial} The WSOS polynomial $-g(\vx)$ corresponding to a (pseudo-moment) vector $\vx\in\Ssc$ is also known as the \emph{Christoffel-Darboux polynomial}, or \emph{inverse Christoffel function}. Recent studies have focused on the properties of this polynomial, and especially on connections between the representing measures of $\vx$ and the sublevel sets of $-g(\vx)$, with applications to design of experiments \cite{DeCastro-etal2021, Lasserre2022}. For our work, the critical property of the Christoffel-Darboux polynomial is the surprising fact that this polynomial is not only WSOS, but that the gradient map $\vx \mapsto -g(\vx) = \Lambda^*(\Lambda(\vx)^{-1})$ yields an explicit WSOS representation of this polynomial: the inverse moment matrix $\Lambda(\vx)^{-1}$ is a Gram matrix that proves that $-g(\vx)$ belongs to $\Sigma$. The concept of dual certificates can be seen as a generalization of this idea: rather than mapping $\vx$ to the Christoffel-Darboux polynomial with an explicit WSOS representation via $-g$, we can map $\vx$ to explicit WSOS representations of a full-dimensional cone of WSOS polynomials $\vs$ via the $(\vx,\vs)\to \vS$ map in \eqref{eq:S-def}.
}

\paragraph{Representation-dependent bit sizes}
Although the main theorem (Theorem~\ref{thm:bitsize-grad}) provides a bit size bound using an array of unconventional parameters, it is worth noting that each of those parameters is easily computable or, in the case of $\cond(\vM)$, can be bounded easily. Although computing the matrix $\vM$ with the lowest condition number is in general a likely impossible task, any unisolvent point set and weight vector in \eqref{eq:mdef} can be used to compute a bound. In each of the special cases considered, it was easy to find a point set that either yields a small enough $\cond(\vM)$ that is dominated by other terms, or one that is provably of the optimal order of magnitude.

The other nontrivial ingredient is the constant $k_1^\ve$ defined in \eqref{eq:k1def}. Unlike the other parameters in Theorem~\ref{thm:bitsize-grad}, it depends not only on the cone $\Sigma$, but its chosen representation via the $\Lambda$ operator. (Equivalently, in the notation of Proposition~\ref{thm:Nesterov}, in depends both on the choice of the $\vp_i$ bases and the $\vq$ basis.) The example of interpolants suggests that it is a measure of conditioning, underscoring the fact that conditioning is consequential even in the case of exact-arithmetic algorithms, not only for numerical methods. Indeed, choosing poor interpolation points (say, equispaced points) to represent WSOS polynomials instead of well-conditioned ones leads to an increase in the bit sizes of the certificates, even in the univariate case, due to the astronomical Lebesgue constant that grows exponentially with the degree, dominating all the other terms in \eqref{eq:interp-univ-bitsize}. It is also this parameter, along with $\nu$, that may be reduced when the polynomials of interest have special structures such as symmetry or term- or correlative sparsity \cite{WangMagronLasserreMai2020}, showing that these structures are useful even for dual certificates.

Also note that none of these parameters need to be known in order to implement the algorithms discussed in Section \ref{sec:computing}, except for the stopping criterion of Algorithm \ref{alg:Newton2}. If we drop the requirement that the algorithm must stop when the returned bound is provably within $\varepsilon$ from the optimal lower bound, and instead run the algorithm until the progress is below a tolerance, or when the lower bound is suitably high (e.g., in applications where the goal is to prove that the input polynomial has a positive WSOS lower bound), then none of the parameters introduced in Theorem \ref{thm:bitsize-grad} need to be computed or bounded.

\paragraph{Dependence on $\nu$} The possibly most counterintuitive aspect of the bound in Theorem~\ref{thm:bitsize-grad} is that the bit size of the integer dual certificate $\overline{\vy}$ (approximated by $U\log(\|\overline{\vy}\|_\infty$)) depends on the $\Sigma$-dependent parameter $\nu$ logarithmically, rather than linearly, since (all else being equal) $\nu$ is a linear function of the number of weights $m$ (recall the notation from \eqref{eq:Swdef} and its surrounding paragraph). For example, consider a family of polyhedral cones of nonnegative polynomials that consist of nonnegative linear combinations of polynomials that are nonnegative on $S_\vw$. This is an elementary special case of WSOS polynomials, where all ``sum-of-squares'' polynomials are simply nonnegative constants, and where it is meaningful to keep adding additional weights to the representation for an increasingly good inner approximation of the cone of nonnegative polynomials without increasing the ``ambient dimension'' $U=\dim(\Sigma)$. It is clear that the sizes of conventional WSOS certificates will grow linearly as a function of $m$: an explicit WSOS decomposition will have $m$ terms; the semidefinite matrix in the representation \eqref{eq:Nesterov-Lambdastar} will have $m$ (one-by-one) semidefinite blocks $\vS_i$, etc. Yet, the dual certificate will remain a $U$-dimensional vector whose components are of size $\Oh(\log(m))$. Of course, this does not mean that such a certificate could be verified in polynomial time for an exponentially large $m$ (like in a Schm{\"u}dgen-type WSOS certificate); verifying that $\vx$ certifies $\vs$, that is, $\Lambda(H(\vx)^{-1}\vs)\succcurlyeq 0$ still requires linear time in the number of weights.

\revision{
\section*{Acknowledgements} The authors are grateful for the referees' thoughtful comments on the presentation, helping to make the paper more accessible to the symbolic computing and computational real algebraic geometry audience. DP would also like to thank Didier Henrion for pointing out the connection between gradient certificates and the Christoffel-Darboux polynomial, and Ali Mohammad Nezhad for pointing us to the results on the bit sizes of subresultants used in the proof of Corollary \ref{cor:univariate-bit-ndtau}.
}

\bibliographystyle{elsarticle-num} 
\bibliography{citations}

\begin{thebibliography}{10}
\expandafter\ifx\csname url\endcsname\relax
  \def\url#1{\texttt{#1}}\fi
\expandafter\ifx\csname urlprefix\endcsname\relax\def\urlprefix{URL }\fi
\expandafter\ifx\csname href\endcsname\relax
  \def\href#1#2{#2} \def\path#1{#1}\fi

\bibitem{Powers2011}
V.~Powers, Rational certificates of positivity on compact semialgebraic sets,
  Pacific Journal of Mathematics 251~(2) (2011) 385--391.
\newblock \href {https://doi.org/10.2140/pjm.2011.251.385}
  {\path{doi:10.2140/pjm.2011.251.385}}.

\bibitem{MagronSafeyElDinSchweighofer2019}
V.~Magron, M.~Safey El~Din, M.~Schweighofer, Algorithms for weighted sum of
  squares decomposition of non-negative univariate polynomials, Journal of
  Symbolic Computation 93 (2019) 200--220.
\newblock \href {https://doi.org/10.1016/j.jsc.2018.06.005}
  {\path{doi:10.1016/j.jsc.2018.06.005}}.

\bibitem{MagronSafeyElDin2021-corrected}
V.~Magron, M.~Safey El~Din, \href{https://arxiv.org/abs/1811.10062v4}{{On Exact
  Reznick, Hilbert-Artin and Putinar's Representations}}, arXiv preprint
  arXiv:1912.04718 (2021).
\newline\urlprefix\url{https://arxiv.org/abs/1811.10062v4}

\bibitem{BoudaoudCarusoRoy2008}
F.~Boudaoud, F.~Caruso, M.-F. Roy, Certificates of positivity in the
  {B}ernstein basis, Discrete and Computational Geometry 39~(4) (2008)
  639--655.
\newblock \href {https://doi.org/10.1007/s00454-007-9042-x}
  {\path{doi:10.1007/s00454-007-9042-x}}.

\bibitem{MagronSafeyElDin2021}
V.~Magron, M.~Safey El~Din, {On Exact Reznick, Hilbert-Artin and Putinar's
  Representations}, Journal of Symbolic Computation 107 (2021) 221--250.

\bibitem{DavisPapp2022}
M.~M. Davis, D.~Papp, Dual certificates and efficient rational sum-of-squares
  decompositions for polynomial optimization over compact sets, SIAM Journal on
  Optimization 32~(4) (2022) 2461--2492.
\newblock \href {https://doi.org/10.1137/21M1422574}
  {\path{doi:10.1137/21M1422574}}.

\bibitem{KatthanNaumannTheobald2021}
L.~Katth\"{a}n, H.~Naumann, T.~Theobald,
  \href{https://doi.org/10.1090/mcom/3607}{A unified framework of {SAGE} and
  {SONC} polynomials and its duality theory}, Mathematics of Computation
  90~(329) (2021) 1297--1322.
\newblock \href {https://doi.org/10.1090/mcom/3607}
  {\path{doi:10.1090/mcom/3607}}.
\newline\urlprefix\url{https://doi.org/10.1090/mcom/3607}

\bibitem{Papp2023}
D.~Papp, \href{https://doi.org/10.1016/j.jsc.2022.04.015}{Duality of sum of
  nonnegative circuit polynomials and optimal {SONC} bounds}, Journal of
  Symbolic Computation 114 (2023) 246--266.
\newblock \href {https://doi.org/10.1016/j.jsc.2022.04.015}
  {\path{doi:10.1016/j.jsc.2022.04.015}}.
\newline\urlprefix\url{https://doi.org/10.1016/j.jsc.2022.04.015}

\bibitem{PappYildiz2019}
D.~Papp, S.~Y{\i}ld{\i}z, Sum-of-squares optimization without semidefinite
  programming, SIAM Journal on Optimization 29~(1) (2019) 822--851.
\newblock \href {https://doi.org/10.1137/17M1160124}
  {\path{doi:10.1137/17M1160124}}.

\bibitem{Nesterov2000}
Y.~Nesterov, Squared functional systems and optimization problems, in:
  H.~Frenk, K.~Roos, T.~Terlaky, S.~Zhang (Eds.), High performance
  optimization, Vol.~33 of Applied Optimization, Kluwer Academic Publishers,
  Dordrecht, 2000, pp. 405--440.
\newblock \href {https://doi.org/10.1007/978-1-4757-3216-0_17}
  {\path{doi:10.1007/978-1-4757-3216-0_17}}.

\bibitem{Laurent2009}
M.~Laurent, Sums of squares, moment matrices and optimization over polynomials,
  in: M.~Putinar, S.~Sullivant (Eds.), Emerging Applications of Algebraic
  Geometry, Vol. 149 of {IMA} Volumes in Mathematics and its Applications,
  Springer, New York, NY, 2009, pp. 157--270.
\newblock \href {https://doi.org/10.1007/978-0-387-09686-5_7}
  {\path{doi:10.1007/978-0-387-09686-5_7}}.

\bibitem{BoydVandenberghe2004}
S.~P. Boyd, L.~Vandenberghe, Convex Optimization, Cambridge University Press,
  2004.

\bibitem{Lasserre2022}
J.~B. Lasserre, A disintegration of the {C}hristoffel function, Comptes Rendus
  Math{\'e}matique 360 (2022) 1071--1079.
\newblock \href {https://doi.org/10.5802/crmath.380}
  {\path{doi:10.5802/crmath.380}}.

\bibitem{DeCastro-etal2021}
Y.~D. Castro, F.~Gamboa, D.~Henrion, J.~B. Lasserre, Dual optimal design and
  the {C}hristoffel–{D}arboux polynomial, Optimization Letters 15 (2021)
  3--8.
\newblock \href {https://doi.org/10.1007/s11590-020-01680-2}
  {\path{doi:10.1007/s11590-020-01680-2}}.

\bibitem{BlekhermanParriloThomas2013}
G.~Blekherman, P.~A. Parrilo, R.~R. Thomas (Eds.), Semidefinite optimization
  and convex algebraic geometry, {SIAM}, Philadelphia, PA, 2013.

\bibitem{Krylov1959}
V.~I. Krylov, Approximate Calculation of Integrals, {Dover}, Mineola, {NY},
  2005.

\bibitem{Beckermann2000}
B.~Beckermann, The condition number of real {V}andermonde, {K}rylov and
  positive definite {H}ankel matrices, Numerische Mathematik 85~(4) (2000)
  553--577.
\newblock \href {https://doi.org/10.1007/PL00005392}
  {\path{doi:10.1007/PL00005392}}.

\bibitem{BasuPollackRoy2006}
S.~Basu, R.~Pollack, M.~Roy, Algorithms in Real Algebraic Geometry,
  {Springer-Verlag Berlin Heidelberg}, 2006.

\bibitem{Trefethen2013}
L.~N. Trefethen, Approximation Theory and Approximation Practice, {SIAM},
  Philadelphia, {PA}, 2013.

\bibitem{MasonHandscomb2003}
J.~C. Mason, D.~C. Handscomb, Chebyshev Polynomials, CRC Press, Boca Raton,
  {FL}, 2003.

\bibitem{BasuLeroyRoy2009}
S.~Basu, R.~Leroy, M.~Roy, A bound on the minimum of a real positive polynomial
  over the standard simplex, arXiv (2009).
\newblock \href {http://arxiv.org/abs/0902.3304} {\path{arXiv:0902.3304}}.

\bibitem{HardyLittlewoodPolya1934}
G.~H. Hardy, J.~E. Littlewood, G.~P\'olya, Inequalities, Cambridge University
  Press, London, 1934.

\bibitem{Wilf1970}
H.~S. Wilf, Finite sections of some classical inequalities, {Springer-Verlag
  Berlin }, 1970.

\bibitem{Papp2017}
D.~Papp, Semi-infinite programming using high-degree polynomial interpolants
  and semidefinite programming, SIAM Journal on Optimizaton 27~(3) (2017)
  1858--1879.
\newblock \href {https://doi.org/10.1137/15M1053578}
  {\path{doi:10.1137/15M1053578}}.

\bibitem{WangMagronLasserreMai2020}
J.~Wang, V.~Magron, J.~B. Lasserre, N.~H.~A. Mai,
  \href{https://arxiv.org/abs/2005.02828}{{CS-TSSOS}: Correlative and term
  sparsity for large-scale polynomial optimization}, technical report. (2020).
\newline\urlprefix\url{https://arxiv.org/abs/2005.02828}

\bibitem{Renegar2001}
J.~Renegar, A mathematical view of interior-point methods in convex
  optimization, MOS-SIAM Series on Optimization, Society for Industrial and
  Applied Mathematics (SIAM), Philadelphia, PA, 2001.
\newblock \href {https://doi.org/10.1137/1.9780898718812}
  {\path{doi:10.1137/1.9780898718812}}.

\bibitem{PappYildiz2017corrigendum}
D.~Papp, S.~Y{\i}ld{\i}z, \href{https://arxiv.org/abs/1712.00492}{On ``{A}
  homogeneous interior-point algorithm for non-symmetric convex conic
  optimization"}, arXiv preprint arXiv:1712.00492 (2017).
\newline\urlprefix\url{https://arxiv.org/abs/1712.00492}

\end{thebibliography}

\appendix
\renewcommand{\thesection}{\Alph{section}}

\section{Appendix}

Here, we summarize relevant results used throughout the paper concerning the local norms $\|\cdot\|_\vx$ and $\|\cdot\|^*_\vx$ and barrier functions of the form $f = -\ln(\det(\Lambda(\cdot)))$ defined on $\Ssc$, introduced in Section \ref{sec:preliminaries}. 

\begin{lemma}\label{thm:f-properties} 
Using the notation introduced in Section \ref{sec:preliminaries}, the following hold for every $\vx\in\Ssc$:
\begin{enumerate}
    \item We have $B_{\vx}(\vx,1) \subset \left(\Sigma^*\right)^\circ$, and for all $\vu\in B_{\vx}(\vx,1)$ and $\vv \neq 0$, one has
    \begin{equation}\label{eq:self-concordance}
    1 - \|\vu - \vx\|_\vx \leq \frac{\|\vv\|_\vu}{\|\vv\|_\vx}\leq (1 - \|\vu - \vx\|_\vx)^{-1}.
    \end{equation}
    \item For all $\vv \neq 0$, if $\|\vu - \vx\|_{\vx} < 1$, we have
    \begin{equation}\label{eq:lemma4}
    \frac{\|\vv\|_{\vu}^*}{\|\vv\|_{\vx}^*} \leq \frac{1}{1-\|\vu - \vx\|_{\vx}}.
    \end{equation}
\item \label{item:gH-prop} The gradient $g$ of $f$ can be computed as
    \begin{equation}\label{eq:g}
    g(\vx) = -\Lambda^*(\Lambda(\vx)^{-1}),
    \end{equation}
    and the Hessian $H(\vx)$ is the linear operator satisfying
    \begin{equation}\label{eq:H}
    H(\vx) \vv= \Lambda^*\!\left(\Lambda(\vx)^{-1}\Lambda(\vv)\Lambda(\vx)^{-1}\right) \;\; \text{ for every } \vv \in \rr^U.
    \end{equation}
   
\item \label{item:loghom} The function $f$ is \emph{logarithmically homogeneous}; that is,
    \[f(\alpha\vx) = f(\vx) - \nu \ln(\alpha) \text{ for every }\alpha > 0\]
    where $\nu = \sum_{i=1}^m L_i$ is the \emph{barrier parameter} of $f$.
    Subsequently, the derivatives of $f$ have the following homogeneity properties:
    \begin{equation}\label{eq:log-homogeneity}
    g(\alpha \vx) = {\alpha}^{-1} g(\vx) \;\text{ and }\; H(\alpha \vx) = {\alpha}^{-2}H(\vx) \;\; \text{ for every } \alpha > 0.
    \end{equation}
    Furthermore,
    \begin{equation}\label{eq:gH-identities}
    H(\vx)\vx = -g(\vx)\quad\text{and}\quad \|g(\vx)\|_{\vx}^* = \|\vx\|_{\vx} = \sqrt{\langle -g(\vx),\vx\rangle} = \sqrt{\nu},
    \end{equation}
    where $\nu$ is the aforementioned barrier parameter.
\item \label{item:bijection} The gradient map $g$ defines a bijection between $\left(\Sigma^*\right)^\circ$ and ${\Sigma}^\circ$, In particular, for every $\vs \in {\Sigma}^\circ$ there exists a unique $\vx \in \left(\Sigma^*\right)^\circ$ satisfying $\vs = -g(\vx)$.
\deletethis{      \item If $\|\vu-\vx\|_\vx < 1$, then
    \begin{equation}\label{eq:lemma5} 
        \|g(\vu) - g(\vx)\|_{\vx}^* \leq \frac{\|\vu - \vx\|_\vx}{1 - \|\vu - \vx\|_\vx}.
    \end{equation}
    \item If $\|g(\vu) - g(\vx)\|_{\vx}^* < 1$, then 
    \begin{equation}\label{eq:revlemma5}
        \|\vu - \vx\|_{\vx}\leq \frac{\|g(\vu) - g(\vx)\|_{\vx}^*}{1 - \|g(\vu) - g(\vx)\|_\vx^*}.
    \end{equation}}
\end{enumerate}
\end{lemma} 

\begin{proof} Statement 1 is Renegar's definition of self-concordance \cite[Sec.~2.2.1]{Renegar2001} to the function $f$ defined in \eqref{eq:def:f}. Statement 2 is \cite[Lemma~4]{PappYildiz2017corrigendum}. Statements \ref{item:gH-prop} and \ref{item:loghom} follow from calculus. Statement 5 is from \cite[Sec.~3.3]{Renegar2001} .
\deletethis{\be 
\item This is Renegar's definition of self-concordance applied to the function $f$, which is a composition of an affine function and a well-known self-concordant function, and is thus self-concordant; see \cite[Sec.~2.2.1 and Thm.~2.2.7]{Renegar2001}. 
\item Straightforward calculation.
\item Straightforward calculation using the identities \eqref{eq:g} and \eqref{eq:H}. We remark that these identities hold for all LHSCBs \cite[Thm.~2.3.9]{Renegar2001}.
\item See \cite[Sec.~3.3]{Renegar2001}.
\deletethis{\item See \cite[Lemma~5]{PappYildiz2017corrigendum}.
\item This is an application of the previous claim to the conjugate barrier function of $f$.}
\item See \cite[Lemma~4]{PappYildiz2017corrigendum}.\qedhere
\ee }
\end{proof}

\end{document}